\documentclass{amsart}

\usepackage{amssymb}
\usepackage{amsfonts}
\usepackage{mathrsfs}

\usepackage[margin=2.95cm]{geometry}

\newcommand{\teich}{Teichm\"uller }

\newcommand{\cftcat}{\ensuremath{\mathcal{C} }}

\newcommand{\disk}{\ensuremath{\mathbb{D}} } 
\newcommand{\cdisk}{\ensuremath{\overline{\mathbb{D}}}} 
\newcommand{\sphere}{\hat{\Bbb{C}}} 
\newcommand{\riem}{\Sigma}  
\renewcommand{\Bbb}[1]{\ensuremath{\mathbb{#1}}}
\newcommand{\st}{\, | \,} 
\newcommand{\ann}{\Bbb{A}}  

\newcommand{\Oqc}{\mathcal{O}^{\mathrm{qc}}} 
\newcommand{\qs}{\operatorname{QS}}

\newcommand{\aoneinfinity}{A_1^\infty(\mathbb{D})}
\newcommand{\aonetwo}{A_1^2(\mathbb{D})}

\newcommand{\Oqco}{\Oqc_{\mathrm{WP}}}

\newcommand{\qso}{\operatorname{QS}_{\mathrm{WP}}}
\newcommand{\Chat}{\hat{\Bbb{C}}}

\newcommand{\tmod}{\operatorname{Mod}}

\newtheorem{theorem}{Theorem}[section]

\theoremstyle{definition}
\newtheorem{definition}[theorem]{Definition}

\theoremstyle{remark}
\newtheorem{remark}[theorem]{Remark}

\numberwithin{equation}{section}

\begin{document}

\title[Quasiconformal Teichm\"uller theory and CFT]{Quasiconformal Teichm\"uller theory as an analytical foundation for two-dimensional conformal field theory}

\author{David Radnell}
\address{David Radnell \\ Department of Mathematics and Systems Analysis \\
Aalto University \\
P.O. Box 11100,  FI-00076 Aalto, Finland }
\email{david.radnell@aalto.fi}

\author{Eric Schippers}
\address{Eric Schippers \\ Department of Mathematics \\
University of Manitoba\\
Winnipeg, Manitoba \\  R3T 2N2 \\ Canada}
\email{eric\_schippers@umanitoba.ca}
\thanks{}

\author{Wolfgang Staubach}
\address{Wolfgang Staubach\\ Department of Mathematics\\
Uppsala University\\
Box 480\\ 751 06 Uppsala\\ Sweden}
\email{wulf@math.uu.se}


\date{\today}

\begin{abstract}
The functorial mathematical definition of conformal field theory was first formulated approximately 30 years ago. The underlying geometric category is based on the moduli space of Riemann surfaces with parametrized boundary components and the sewing operation. We survey the  recent and careful study of these objects, which has led to significant connections with quasiconformal
 \teich theory and geometric function theory.

In particular we propose that the natural
 analytic setting for conformal field theory
 is the moduli space of Riemann surfaces with so-called Weil-Petersson
 class parametrizations.  A collection
 of rigorous analytic results is advanced here as evidence.
  This class of parametrizations has the required regularity for CFT on one hand, and on the other hand are natural and of interest in their own right in geometric function theory.
\end{abstract}

\thanks{Eric Schippers and Wolfgang Staubach are grateful for the financial support from the Wenner-Gren Foundations. Eric Schippers is also partially supported by the National Sciences and Engineering Research Council of Canada. }

\maketitle

\begin{section}{Introduction} \label{Introduction}

\begin{subsection}{Introduction}

Two-dimensional \emph{conformal field theory} (CFT) models a wide range of physical phenomena.  Its mathematical structures connect to many branches of mathematics, including complex geometry and analysis, representation theory, algebraic geometry, topology, and stochastic analysis. There are several mathematical notions of CFT, each of which is relevant to probing the particular mathematical structures one is interested in. Without attempting any overview of this vast subject, we first highlight some early literature, and then explain the purpose of this review.\bigskip

The conformal symmetry group in two-dimensional quantum field theory goes back to at least the Thirring model from 1958, although this was perhaps not fully recognized immediately. The role of conformal symmetry was expanded upon in the 1960's and appeared in string theory in the 1970's. At the same time in string theory the moduli space of Riemann surfaces also appeared via the world sheets of strings. The rich algebraic structure of CFT led to solvable models in physics.  The development of CFT as an independent field was in part due to the seminal work of A.~A. Belavin, A.~M. Polyakov, and A.~B. Zamolodchikov \cite{BPZ} from 1984 and the development by R. E. Borcherds \cite{Borcherds} and I. B. Frenkel, J. Lepowsky and A. Meurman \cite{frenklepowskymeurman} of vertex operator algebras in mathematics.  Geometric formulations of CFT were developed by 1987 based on the complex analytic geometry of the moduli space of Riemann surfaces. See for example \cite{FS, Segal87, SegalPublished, Vafa1987} and references therein. G. Segal \cite{Segal87, SegalPublished} and M. Kontsevich sketched a mathematically rigorous axiomatic definition of CFT, highlighting the geometric operation of sewing Riemann surfaces using boundary parametrizations and the algebraic structure this produces (see Section \ref{se:CFT}). A very similar approach using surface with punctures and local coordinates was developed by C. Vafa \cite{Vafa1987}. \bigskip

We will be concerned here with the functorial definition of G. Segal \cite{SegalPublished} and the ongoing program to rigorously construct CFTs from vertex operator algebras (VOAs) carried out by Y.-Z. Huang and others (see for example \cite{Huang_91, Huang_97, Huang,  Huang_Diff_Eq, Huang_Diff_Int, HuangKong_07}). Of particular interest are the related notions of   \emph{ $($holomorphic$)$ modular functors} and \emph{ $($holomorphic$)$ weakly conformal field theories} formulated by Segal \cite{SegalPublished}. These objects encode the rich mathematical structure of chiral CFT, and their construction is the first step in constructing full CFTs from VOAs. See \cite{Huang_CFT} for further explanations and an overview of the mathematical development of CFT. Moreover Segal's definition of CFT involves holomorphicity
properties of infinite-dimensional moduli spaces and the sewing operation, which need to be established rigorously. As a consequence, we have decided to focus here
on pinpointing the correct analytic and geometric setting for these infinite-dimensional moduli spaces, and clarifying the connections of the problems arising in this context with quasiconformal Teichm\"uller theory.\bigskip

We should also briefly explain the context of this paper within the existing literature. Many topological and geometric
questions in CFT can be investigated without dealing with the
infinite-dimensional spaces. For instance, as was mentioned in \cite[page 198]{Vafa1987}: ``\emph{The infinite-dimensional space $[$just mentioned$]$ may at first
sight seem too complicated to deal with. But it turns out that all topological questions about it can be reduced to questions about a finite-dimensional space
$\tilde{P}(g,n)$ which is the moduli of Riemann surfaces with ordered punctures and a choice of a non-vanishing tangent vector at each point.}"  Furthermore, numerous rigorous mathematical works have studied $\tilde{P}(g,n)$ from the point of view of topology, algebraic geometry, category theory and \teich theory and have made far reaching conclusions. \bigskip

Nevertheless, many deep problems of CFT
involve the infinite-dimensional moduli spaces, see e.g. \cite{Huang_91, Huang,  Huang_CFT} and reference therein.
To our knowledge, apart from the work of the authors \cite{Radnell_thesis, RadnellSchippers_monster, RSS_Filbert2} and that of Y.-Z. Huang \cite{Huang} in genus zero and K. Barron \cite{BarronN1, BarronN2} in the genus zero super case, neither the analytic structure of these infinite-dimensional moduli spaces, nor even their point-set topology, have been rigorously defined or studied.
The success of the construction program of CFTs from VOAs in genus zero and one, and the profound new discoveries that it entailed, makes the infinite-dimensional moduli spaces both unavoidable and of great continued interest in the higher genus case. Its study motivates the introduction of quasiconformal \teich spaces into CFT, and has influenced our own work which brings together tools and ideas from CFT, geometric function theory, Teichm\"uller theory and hard analysis.\bigskip

The analytic structure of these infinite-dimensional moduli spaces depends on the class or boundary parametrizations used. There is a subtle interplay between the regularity of the parametrization, regularity of the boundary curve, and the function spaces used to model the moduli spaces. \bigskip

Our work shows that the so-called Weil-Petersson class boundary parametrizations
provide the correct analytic setting for conformal field theory.  Furthermore, using the larger class of quasisymmetric boundary parametrizations (equivalently, quasiconformally extendible conformal
maps in the puncture picture) draws a clear connection between the moduli space appearing in
CFT and quasiconformal Teichm\"uller space.  This leads to significant insight into the
moduli space appearing in CFT, which can be exploited to resolve many longstanding analytic and geometric
conjectures.\bigskip

The original and ongoing motivation for our work is threefold. The first is that the definitions of  \emph{holomorphic modular functors} and \emph{weakly conformal field theories} themselves rely on a number of analytic and geometric conjectures. We have recently rigorously formulated and proved some of these conjectures with  the help of quasiconformal \teich theory. In doing so, we have several aims. One is to develop a natural and useful analytic foundation for these definitions on which further geometric structures can be defined and studied. Another is to contribute to the program of constructing CFT from vertex operator algebras. Completion of the higher-genus theory not only requires these rigorous definitions, but also certain analytic results which can only be addressed within the quasiconformal \teich setting.  A third is the use of ideas from CFT to obtain new results which are of interest in \teich theory and geometric function theory. More generally the aim is to uncover and make explicit the deep connections between these fields.\bigskip

We conclude this section with some observations.  Ten years ago we showed that
 the rigged moduli space is a quotient of Teichm\"uller space by a discrete group.
 Thus the rigged moduli space has been studied
 in Teichm\"uller theory, in a different form, for decades before it appeared in conformal field theory.
 This has tremendous implications for both fields beyond the resolution of
 analytic issues. Simply translating some of the ideas from one field to
 the other, leads to significant insights in both fields.  Many of the analytic
 issues in CFT have already been resolved; at worst,
 the setting for their resolution is now in place. Moreover the insights arising from the interconnection between Teichm\"uller theory and CFT are of utmost importance in solving certain problems in CFT, e.g. sewing properties of meromorphic functions on moduli spaces (see Y.-Z. Huang's article in this proceedings) and a description
 of the determinant line bundle (see Section \ref{se:analytic_setting_detline}).
 We hope that this paper will entice researchers in CFT and Teichm\"uller theory to explore the rich connections between these fields.

\end{subsection}

\begin{subsection}{Organization of the paper}
 In Section \ref{se:CFT}, we sketch the definition of two-dimensional conformal field theory and related notions and
 extract the analytic requirements of these definitions.
 We only provide a very general overview; the reader can consult \cite{SegalPublished} for more
 on this question.  We also define the rigged moduli space, and introduce the question of which
 choice of riggings, which is the theme of this paper.  In Section \ref{se:Teich_theory}
 we outline the basic ideas and results of quasiconformal Teichm\"uller theory.
 In Section \ref{se:Teich_RMS_correspondence}  we sketch our results (in various combinations
 of authorship) on the correspondence between Teichm\"uller space and the rigged moduli space.
 We also define the sewing operation and outline our results on holomorphicity.  Finally,
 we define the Weil-Petersson class Teichm\"uller space in genus $g$ with $n$ boundary curves.
 In Section \ref{se:analytic_setting_detline} we define the decompositions of Fourier series
 which appear in the definition of conformal field theory, and define the determinant line
 bundle.  We discuss our results on these decompositions in the context of quasisymmetric
 riggings, our results on the jump formula for quasidisks and Weil-Petersson class quasidisks,
 and how they show that the decompositions hold in the case of quasisymmetric riggings.
 Finally, we give some new results that outline the connection of the operator $\pi$ (one operator through which the
 determinant line bundle can be defined) to the Grunsky operator in geometric function theory,
 and illustrate the relevance of the Weil-Petersson class for the determinant line bundle. In the final section we review the
 case for quasisymmetric or Weil-Petersson class riggings.
\end{subsection}
\end{section}
\begin{section}{Conformal field theory}  \label{se:CFT}
\begin{subsection}{Notation} \label{se:notation}
Let $\Chat$ be the Riemann sphere, $\disk = \disk^+ = \{z \in \mathbb{C} \st |z|<1\}$, and $\disk^- = \Chat \setminus \overline{\disk}$.  For $r, s \in \mathbb{R}_{+}$ with $r<s$,
let  $\ann_r^s = \{ z \st r<|z| < s \}$.
\end{subsection}
\begin{subsection}{The rigged moduli space of CFT}\label{riggedmoduli}
 In this section we give a definition of the moduli space of rigged Riemann surfaces arising
 in CFT \cite{SegalPublished, Vafa1987}.  We then give a preliminary discussion on the ramifications of different choices of
 analytic categories in this definition.  Throughout this section, the term ``conformal''
 denotes one-to-one holomorphic maps (as opposed to locally one-to-one).


 A Riemann surface is a complex manifold of complex-dimension one.
 Let $g,k,l$ be positive integers.  We say that a Riemann surface $\riem$ is of type $(g,k,l)$
 if it has $g$ handles, the boundary contains $k$ punctures and has $l$ borders homeomorphic to the circle, and no other boundary points. Although punctures and borders are not distinguishable topologically, they are distinguishable
 holomorphically: for example, by Liouville's theorem the sphere with one puncture is homeomorphic but not biholomorphic to the disk.

 Furthermore the term
 ``border'' has a precise meaning \cite{Ahlfors_Sario}.  For our purposes, it
  suffices to say that the Riemann surface has a double, and its boundary is an analytic curve in the
 double.
 Let $\partial_i \riem$ be a border homeomorphic to $\mathbb{S}^1$.
 A ``collar chart'' is a biholomorphism $\zeta_i:A_i \rightarrow \mathbb{A}_1^r$ where $r>1$ and
 $A_i$ is an open set in $\riem$ bounded on one side by $\partial_i \riem$ and on the other
 by a simple closed analytic curve in $\riem$ homotopic to $\partial_i \riem$. We furthermore
  require that the chart has a homeomorphic extension to $\partial_i \riem$
 such that $\zeta_i(\partial_i \riem) = \mathbb{S}^1$.
 By Schwarz reflection the chart has a conformal extension
 to an open neighbourhood of the boundary curve $\partial_i \riem$ in the double.  We will
 not have use for the double except in this definition; however, the extension to the
 boundary curve $\partial_i \riem$ will be used frequently without comment.

\begin{definition} Let $\riem$ be a Riemann surface.
\begin{enumerate}
\item $\riem$ is a \emph{bordered Riemann surface of type $(g, n)$} if it is of type $(g, 0, n)$.   Such surfaces will be denoted $\riem^B$.
\item $\riem$ is a \emph{punctured Riemann surface of type $(g, n)$} if it is of type $(g, n, 0)$. Such surfaces will be denoted $\riem^P$ and the punctures by $p_1, \ldots, p_n$.
\item A connected component of the boundary of a bordered Riemann surface $\riem^B$ is called a \emph{boundary curve}. Denote the boundary curves by $\partial_i \riem^B$ for $i=1,\ldots, n$. Note that each boundary curve is homeomorphic to $\mathbb{S}^1$. For a punctured surface, $\partial_i \riem^P = p_i$.
\end{enumerate}
\end{definition}
 We will often treat the punctures of a Riemann surface as points in the Riemann surface,
 and similarly for the borders.
 \begin{remark} Every class of mappings considered in this paper has a unique
  continuous extension to the punctures or borders (whichever is relevant).
  This is an important technical
  point, but we will not repeat it in each special case.
 \end{remark}

 There are two basic models of the rigged moduli space: as a collection of punctured
 Riemann surfaces with specified conformal maps onto neighborhoods of the punctures \cite{Vafa1987}, or as a collection of bordered Riemann
 surfaces with boundary parametrizations \cite{SegalPublished}.  We will refer to these as the puncture
 and border model respectively.  In both cases, we will call the extra data (conformal
 maps or parametrizations) ``riggings''.   For now we will purposefully not specify
 the analytic category of the riggings.

 {\bf Puncture model}.  Fix a punctured Riemann surface $\riem^P$ of type $(g,n)$ with
 punctures, $p_1,\ldots,p_n$.  Let $\mathcal{A}^P$ denote a class of maps
 $f:\disk \rightarrow \mathbb{C}$ such that $f(0)=0$ and $f$ is conformal.
 \begin{definition}[Riggings, puncture model]  \label{de:riggings_puncture_model}
  The class of riggings
  $\mathcal{R}(\mathcal{A}^P,\riem^P)$ is the collection of $n$-tuples of conformal maps
  $(f_1,\ldots,f_n)$ such that
  (1) $f_i:\disk \rightarrow \riem^P$ is conformal, (2) $f_i(0)=p_i$,
  (3) the closures of the images do not overlap, (4) there are local
  biholomorphic coordinates $\zeta_i$ on an open neighbourhood of
  the closure of $f_i(\disk)$ such that $\zeta_i(p_i)=0$ and
  $\zeta_i \circ f_i \in \mathcal{A}^P$.
 \end{definition}
 \begin{remark}  For all the classes $\mathcal{A}^P$ considered in this
  paper, if there is one normalized biholomorphic coordinate $\zeta_i$ at $p_i$ such that
  $\zeta_i \circ f_i \in \mathcal{A}^P$, then every normalized biholomorphic coordinate
  defined on an open neighbourhood of $f_i(\disk)$ has this property.  This is an
  important technical point, but we will take it for granted here.
 \end{remark}
 \begin{definition}[Rigged moduli space, puncture model]
   The $\mathcal{R}(\mathcal{A}^P,\riem^P)$-rigged moduli space of punctured
 Riemann surfaces is
 \[  \widetilde{\mathcal{M}}^P(\mathcal{A}^P) = \{ (\riem^P_1,f) \}/\sim  \]
 where $\riem^P_1$ is a punctured Riemann surface, $f =
 (f_1,\ldots,f_n) \in \mathcal{R}(\mathcal{A}^P,\riem^P)$,
 and
 $  (\riem^P_1,f) \sim (\riem^P_2,g)$
 if and only if there is a conformal bijection $\sigma: \riem^P_1 \rightarrow \riem^P_2$
 such that $g_i = \sigma \circ f_i$ for all $i=1,\ldots,n$.
 \end{definition}
 \begin{remark}
  The condition that $f_i$ are defined on the unit disk $\disk$ might conceivably
  be relaxed, e.g.
  one might consider germs of conformal maps at $0$.  Similarly, the condition
  that the closures of the domains do not intersect can be relaxed.
 \end{remark}

 {\bf Border model}: Let $\riem^B$ be a bordered Riemann
  surface of type $(g,n)$.  Let $\mathcal{A}^B$ be a class of orientation preserving homeomorphisms
  $\phi:\mathbb{S}^1 \rightarrow \mathbb{S}^1$.
 \begin{definition}[Riggings, border model] \label{de:riggings_border_model}
   The class of riggings $\mathcal{R}(\mathcal{A}^B,\riem^B)$ is the collection
 of $n$-tuples of maps $(\phi_1,\ldots,\phi_n)$ such that (1) $\phi_i :\mathbb{S}^1 \rightarrow \partial_i \riem^B$ and (2) there is a collar chart $\zeta_i$ of
  each boundary curve $\partial_i \riem$ such that $\zeta_i \circ \phi_i \in \mathcal{A}^B$ for $i=1,\ldots,n$.
 \end{definition}

With respect to the orientation of $\partial_i \riem$ induced from $\riem$ and the standard orientation of $\mathbb{S}^1$, $\phi_i$ is orientation reversing. While mathematically this choice is arbitrary and of no importance, the choice we made is to match certain conventions in \teich theory. In \cite{RadnellSchippers_monster} the opposite choice was made by using collar charts onto annuli $\ann_r^1$.

 \begin{remark} \label{re:riggings_border_model} Again, for all the choices of $\mathcal{A}^B$
  considered in this paper, if $\zeta_i \circ \phi_i \in \mathcal{A}^B$
  for one collar chart, then it is in $\mathcal{A}^B$ for all of them.
 \end{remark}

 \begin{definition}[Rigged moduli space, border model]
  The rigged moduli space of bordered Riemann surfaces is
 \[  \widetilde{\mathcal{M}}^B(\mathcal{A}^B) =  \{ (\riem^B_1,\phi) \}/\sim  \]
 where $\riem^B_1$ is a bordered Riemann surface, $\phi =(\phi_1,\ldots,\phi_n) \in \mathcal{R}(\mathcal{A}^B,\riem^B)$, and
 $  (\riem^B_1,\phi) \sim (\riem^B_2,\psi) $
 if and only if there is a conformal bijection $\sigma:\riem^B_1 \rightarrow \riem^B_2$
 such that $\psi_i = \sigma \circ \phi_i$ for $i=1,\ldots,n$.
 \end{definition}

 The question is, what are $\mathcal{A}^P$ and $\mathcal{A}^B$?
  This is a deep
 analytic question.  The temptation to dismiss it is easily dispelled if one reflects
 on the analogous example of Fourier series.   $L^2$ spaces allow the unfettered expression
 of the linear algebraic structure of Fourier series. Other choices of regularity
  of the functions (say, piecewise-$C^1$ or smooth)
 lead to theorems which bury simple algebraic ideas in a tangle of qualifications.
 On the other hand, the cost of choosing the setting in which the theorems have
 simple expressions is that the proofs get harder and the definitions more subtle,
 so that the work is done at the beginning.
 Our work is very much of this nature.

 In the next section, we sketch the definition of conformal field theory, and
 then discuss the analytic conditions necessary for the realization of an
 example.
\end{subsection}

\begin{subsection}{Definition of Conformal Field theory}
\label{ss:CFT}

We give very brief outlines of the definitions of CFT, weakly CFT and modular functor as originally given in Segal \cite{SegalPublished}, with commentary on the analyticity requirements. See also \cite{Hu_Kriz05, Huang_97, Huang_CFT}. We follow the expositions in \cite{Huang_CFT, SegalPublished}.

Recall from Subsection 2.2 that a rigged moduli space (border model) element is a conformal equivalence class of Riemann surfaces with parametrized boundary components. In addition to the ordering of the boundary components we now assign an element of $\{+, - \}$ to each boundary component and refer to this choice as the orientation.  Moreover, we allow Riemann surfaces to be disconnected. In this subsection we temporarily use the term ``rigged moduli space'' in this generalized setting; in the remainder of the paper we will assume the Riemann surfaces are connected.

Let $\cftcat$ be the category defined as follows. The objects are finite ordered sets of copies of the unit circle $\mathbb{S}^1$.  Morphisms are elements of the rigged moduli space where the copies of $\mathbb{S}^1$ in the domain (codomain) parametrize the negatively (positively) oriented boundary components. Composition of morphisms is via the sewing operation as defined in Subsection \ref{sewingoperation}. Let $\mathcal{T}$ be the tensor category of complete locally convex topological complex vector spaces with nondegenerate bilinear forms. Morphisms are trace-class operators. A \emph{conformal field theory} is a projective functor from $\cftcat$ to $\mathcal{T}$ satisfying a number of axioms, see e.g. \cite{Huang_CFT}. Chiral CFT refers to the holomorphic and antiholomorphic parts of CFT. Mathematically this is formalized below in the more general notion of a holomorphic weakly conformal field theory.

If $\Phi$ denotes a finite set of labels, let $\mathcal{S}_{\Phi}$ be the category whose objects are rigged moduli space elements where each boundary component of the Riemann surfaces is assigned a label from $\Phi$. Morphisms are given by the sewing operation.
A \emph{holomorphic modular functor} is a holomorphic functor $E$ from  $\mathcal{S}_{\Phi}$ to the category of finite-dimensional vector spaces.

Given a holomorphic modular functor we can extend the category $\cftcat$ to a category $\cftcat_E$ where the morphisms are now pairs $([\riem, \phi], E([\riem, \phi]))$
where $[\riem , \phi]$ is a rigged moduli space element and $E([\riem, \phi])$ is the vector space specified by the modular functor.
A \emph{holomorphic weakly conformal field theory} is a functor $U$ from $\cftcat_E$ to $\mathcal{T}$ satisfying axioms analogous to those in the definition of a CFT. Full CFTs can be constructed from holomorphic weakly conformal field theories.

To make the definitions of a holomorphic modular functor and a holomorphic weakly conformal field theory rigorous, a number of holomorphicity requirements must be addressed, and these justify the relevance of quasiconformal Teichm\"uller in this context.
It is crucial to understand that these requirements are hidden in the abstractness of the term ``holomorphic functor''. Explicitly they are:
\begin{enumerate}
\item The rigged moduli space is an infinite-dimensional complex manifold.
\item The sewing operation is holomorphic.
\item The vector spaces $E[(\riem, \phi)]$ form a holomorphic vector bundle over the rigged moduli space and the sewing operation is also holomorphic on the bundle level.
\item The determinant line bundle (see Section \ref{se:decompositions}) is a one-dimensional example of a modular functor.
\item $U([\riem, \phi], E([\riem, \phi]))$ depends holomorphically on $([\riem, \phi], E([\riem, \phi]))$.
\end{enumerate}

Items (1), (2) and (4) are not assumptions but are facts that must be proven prior to making the above definitions.
Items (3) and (5) must be addressed in the construction of examples of holomorphic modular functors and holomorphic weakly conformal field theories.

To the authors' knowledge the notions of a holomorphic modular functor and a holomorphic weakly conformal field theory for genus greater than zero have never been studied in detail due to these infinite-dimensional holomorphicity issues involving boundary parametrizations. The results under review in this article address items (1) and (2) and more generally provide a natural analytic setting that enables rigorous study of these objects.  Forthcoming papers will address item (4).

\begin{remark}
The original terminology of \cite{SegalPublished} is \emph{modular functor} and \emph{weakly conformal field theory}. Other authors have added the word ``holomorphic"  to make an important distinction from various notions of \emph{topological modular functor} that have appeared in the literature (see for example \cite{AndersonUenoGZ, bakalovkirill, Walker}). To add to the confusion, the term ``complex-analytic modular functor" (used in \cite{bakalovkirill} and by others) is based on the finite-dimensional moduli space $\tilde{P}(g,n)$ mentioned in Section \ref{Introduction}.
\end{remark}

\end{subsection}
\begin{subsection}{Which class of riggings?}
We now return to the question of what analytic conditions to place on the
riggings.

For any choice of $\mathcal{A}^P$, there is a corresponding $\mathcal{A}^B$ which makes
 the rigged moduli spaces of border and puncture type identical.  Thus there is only one
 analytic condition to be chosen, but which one?     Some
 choices in the literature are:
 \begin{enumerate}
  \item  $\mathcal{A}^B$ consists of analytic diffeomorphisms of $\mathbb{S}^1$, and $\mathcal{A}^P$
  consists of maps which are conformal on $\disk$ holomorphic on some open neighbourhood of the
  closure of $\disk$;
  \item  $\mathcal{A}^B$ consists of diffeomorphisms of $\mathbb{S}^1$ and $\mathcal{A}^P$ consists of maps which are conformal on $\disk$ and
      extend diffeomorphically to the closure of $\disk$;
  \item $\mathcal{A}^P$ consists of quasisymmetric homeomorphisms of $\mathbb{S}^1$
   and $\mathcal{A}^B$ consists of conformal maps of $\disk$ with quasiconformal extensions
   to an open neighbourhood of the closure of $\disk$;
  \item $\mathcal{A}^P$ consists of Weil-Petersson class quasisymmetric maps,
  and $\mathcal{A}^B$ consists of Weil-Petersson class quasiconformally
  extendible conformal maps.
 \end{enumerate}
 We will explain the terms ``quasisymmetric'', ``quasiconformal'', and
 ``WP-class'' ahead.
 The first two choices are the most common in conformal field theory.  The last
 two appear in Teichm\"uller theory, but in
 the context of conformal field theory/rigged moduli space, they
 appear almost exclusively in our own work.
 Many authors have suggested a connection between string theory and the universal
  Teichm\"uller space $T(\disk)$ (defined ahead); furthermore, there is a direct connection to representations
  of the Virasoro algebra since $T(\disk)$ can be identified with quasisymmetries of the
   circle modulo M\"obius maps of the circle.  Quasisymmetries and
 quasiconformal extendible maps frequently appear in those contexts.  A (somewhat dated) review of
 some of the literature in this direction can be found in \cite{Pekonen}.
 More recent references include \cite{Pickrell,SergeevArmen}.

 The WP-class quasisymmetries are the correct analytic choice
 for the formulation of CFT.  The quasisymmetries are nevertheless an important class because
 that choice results in a link between quasiconformal Teichm\"uller theory and
 conformal field theory. The implications of these choices will be reviewed
 at the end of this paper, but we give a brief overview now.

  If one chooses $\mathcal{A}^P$ to be the
 quasisymmetries, then (as we will see) the rigged
   moduli space is, up to a discrete group action, equal to the Teichm\"uller
   space of bordered surfaces.  This choice is sufficient to endow the rigged
   moduli space with a (Banach manifold) complex structure, prove
   holomorphicity of sewing, and give holomorphic structures to some
   vector bundles over the rigged moduli space, but not the determinant
   line bundle.  On the other hand, if one chooses $\mathcal{A}^P$ to be
   the WP-class quasisymmetries,
   then we can identify the rigged moduli space (up to a discrete group
   action) with the WP-class Teichm\"uller space.  This choice is sufficient
   to endow the rigged moduli space with a (Hilbert manifold) complex structure,
   prove holomorphicity of sewing (currently work in progress by the authors), give holomorphic structure
   to vector bundles including the determinant line, and show the existence of sections (equivalently, determinants of certain operators) of the determinant line bundle.
	
 The remainder of the paper is dedicated to a careful exposition
 of these statements.

\end{subsection}
\end{section}
\begin{section}{Quasiconformal mappings and \teich theory}  \label{se:Teich_theory}
Here we collect in a concise way some of the ideas of quasiconformal Teichm\"uller theory.


\begin{subsection}{Differentials, quasiconformal mappings and the Beltrami equation}

Quasiconformal Teichm\"uller theory was introduced by L. Bers \cite{bers-65,bers-london,Bers1}. Here we shall review some basic concepts from this theory which are crucial in the applications to CFT.

\begin{definition}
  A \emph{Beltrami differential} on $\Sigma$
 is a $(-1,1)$ differential $\omega$, i.e., a differential given
 in a local biholomorphic coordinate $z$ by $\mu(z)d\bar{z}/dz$, such
 that $\mu$ is Lebesgue-measurable in every choice of coordinate and
 $||\mu||_\infty <1$.  The expression $||\mu||_{\infty}$ is
 well-defined, since $\mu$ transforms under a local biholomorphic change of parameter
 $w=g(z)$  according to the rule
 $\tilde{\mu}(g(z))\overline{g'(z)}g'(z)^{-1}=\mu(z)$ and
 thus $|\tilde{\mu}(g(z))|=|\mu(z)|$.
\end{definition}

Denote the space of Beltrami differentials on $\Sigma$ by $L^\infty_{(-1,1)}(\Sigma)_1$.
The importance of the Beltrami differentials stems from the following fundamental partial differential equation in geometric function theory and Teichm\"uller theory.
\begin{definition}
  Let $\Sigma$ be a Riemann surface.
 The \emph{Beltrami equation} is the differential equation given in local
 coordinates by $\overline{\partial} f = \omega \partial f$
 where $\omega$ is a Beltrami differential.
\end{definition}
In this connection, we have the important theorem, see e.g. \cite{AstalIwaniecMartin}:
 \begin{theorem} \label{th:existuniqBeltrami}
  Given any Beltrami differential on a Riemann surface $\Sigma$,
  there exists a homeomorphism $f:\Sigma \rightarrow \Sigma_1$, onto
  a Riemann surface $\Sigma_1$, which is absolutely continuous
  on lines and is a solution of the Beltrami equation almost everywhere.
  This solution is unique in the sense that given any other solution
  $\tilde{f}: \Sigma \rightarrow \widetilde{\Sigma}_1$, there exists
  a biholomorphism $g:\Sigma_1 \rightarrow \widetilde{\Sigma}_1$
  such that $g \circ f = \tilde{f}$.
 \end{theorem}
 The condition ``absolutely continuous on lines'' means that if $f$ is
 written in local coordinates, in any closed
 rectangle in $\mathbb{C}$, $f$ is absolutely continuous on almost every
 vertical and horizontal line.  In particular, the partial derivatives exist
 almost everywhere and the statement that $f$ satisfies the Beltrami equation
 is meaningful.  It can be further shown that $f$ is differentiable almost
 everywhere.

 If $||\omega||_\infty=0$, then $f$ must be a biholomorphism.
 \begin{definition}\label{defn:quasiconformalmap}
The solutions of the Beltrami equation are called \emph{quasiconformal mappings}.
 \end{definition}

If $f: X  \to X_1$ is quasiconformal then, in terms of a local parameter $z$,
$$\mu(f) = \frac{\partial f}{\partial \bar{z}} / \frac{\partial f}{\partial z}$$
is called the \textit{complex dilation} of $f$. \bigskip

Although there are various equivalent definitions of quasiconformal mappings, for the purposes of this paper we will use Definition \ref{defn:quasiconformalmap}.

 Given a Beltrami differential and the corresponding quasiconformal
 solution to the Beltrami equation $f:\Sigma \rightarrow \Sigma_1$,
 one can pull back the complex structure on $\Sigma_1$ to obtain a
 new complex structure on $\Sigma$.
 Thus, one can regard a Beltrami differential as a change of the complex structure on $\Sigma$.


\end{subsection}
\begin{subsection}{Quasisymmetric maps and conformal welding} \label{se:quasisymmetric}

 In this section we define the quasisymmetries of the circle.
 Quasisymmetries play a central role in Teichm\"uller theory \cite{Lehto, Nagbook}. They arise as boundary values of quasiconformal maps. Recall the notation from Section \ref{se:notation}.

\begin{definition}\label{quasisymmetriccircle}
An orientation-preserving homeomorphism $h$ of $\mathbb{S}^1$ is called a \emph{quasisymmetric mapping}, if there is a constant $k>0$, such that for every $\alpha$ and every $\beta$ not equal to a multiple of $2\pi$, the inequality
 \[  \frac{1}{k} \leq \left| \frac{h(e^{i(\alpha+\beta)})-h(e^{i\alpha})}{h(e^{i\alpha})-h(e^{i(\alpha-\beta)})} \right|
    \leq k \]
holds. Let $\qs(\mathbb{S}^1)$ be the set of quasisymmetric maps from $\mathbb{S}^1$ to $\mathbb{S}^1$.
\end{definition}
A useful property of quasisymmetries is the following:

 \begin{theorem}
   The set of quasisymmetries form a group under composition, see e.g. \cite{Lehto}.
 \end{theorem}

The following theorem, due to A. Beurling and L. Ahlfors \cite{BeurlingAhlfors} explains the importance of quasisymmetric
mappings in Teichm\"uller theory, see also \cite[II.7]{LV}.
\begin{theorem} \label{basicqcextension}
 A homeomorphism $h:\mathbb S^1 \rightarrow
 \mathbb S^1$ is quasisymmetric if and only if there
 exists a quasiconformal map of the unit disk $\mathbb{D}$ with boundary
 values $h$.
\end{theorem}
Note that not every quasisymmetry is a diffeomorphism. Finally,  the following theorem (see \cite[III.1.4]{Lehto}) describes the
classical conformal welding of disks.

\begin{theorem}
\label{th:welding} If $h : \mathbb{S}^1 \longrightarrow \mathbb{S}^1$ is
quasisymmetric then there exists conformal maps $F$ and $G$ from
$\disk^+$ and $\disk^-$ into complementary Jordan domains $\Omega^+$
and $\Omega^-$ of $\Chat$, with quasiconformal extensions to $\Chat$ such that $(G^{-1} \circ F)|_{\mathbb{S}^1} = h$.
Moreover, the Jordan curve separating $\Omega^+$ and $\Omega^-$ is a
quasicircle.  $F$ and $G$ are determined uniquely up to simultaneous post-composition
with a M\"obius transformation.
\end{theorem}
 This theorem is originally due to A. Pfluger \cite{Pfluger}, and is now standard
 in Teichm\"uller theory \cite{Lehto}.
 It follows easily from the existence and uniqueness of solutions to the Beltrami equation. One also sees that in the case that $h$ is a diffeomorphism, $F$ and $G$ in the theorem extend diffeomorphically to an open neighbourhood of the closure of the disk. The conformal welding theorem in the case of diffeomorphisms was also proved independently, some twenty-five years later, by A. A. Kirillov \cite{Kir1}, but he is sometimes (inaccurately) attributed as the originator of the theorem.

\end{subsection}


\begin{subsection}{\teich space}
We summarize here some of the basics of quasiconformal Teichm\"uller theory.  A
 comprehensive treatment can be found for example in the books \cite{Lehto} and \cite{Nagbook}.
 Although we restrict our attention to bordered surfaces of type $(g,n)$ and punctured
 surfaces of type $(g,n)$ to simplify our presentation, the theory applies to the
 Teichm\"uller space of any
 Riemann surface covered by the disk.  The reader should be aware that
 most of the definitions and theorems hold in
 this general case, after a suitable treatment of the boundary and homotopies rel boundary.

Let $\riem$ be a fixed base Riemann surface of border or puncture type $(g,n)$.
Consider the set of triples $(\riem,f_1,\riem_1)$ where $\riem_1$
is a Riemann surface, and $f_1:\riem \rightarrow \riem_1$ is
a quasiconformal mapping. We say that \[ (\riem,f_1,\riem_1)
\sim_T (\riem,f_2,\riem_2) \] if there exists a
biholomorphism $\sigma: \riem_1 \rightarrow \riem_2$ such
that $f_2^{-1} \circ \sigma \circ f_1$ is isotopic to the identity
`rel $\partial \riem$'.  Recall that the term `rel $\partial
\riem$' means that the isotopy is constant on punctures or borders $\partial
\riem$; in particular it is the identity there.

\begin{definition}
 The Teichm\"uller space of a Riemann surface $\riem$
 is
 \[ T(\riem) = \{ (\riem,f_1,\riem_1) \}/\sim_T.  \]
\end{definition}

By Theorem \ref{th:existuniqBeltrami} there is a map
 $\Phi_{\Sigma}:L^\infty_{(-1,1)}(\Sigma)_1 \rightarrow
T(\riem) $ from the space of Beltrami differentials to the Teichm\"uller space, given by mapping a Beltrami differential $\mu
 d\bar{z}/dz$ to the corresponding quasiconformal solution of the
 Beltrami equation. The map $\Phi_{\Sigma}$ is called the \emph{fundamental projection}.

A basic characterising theorem for the Teichm\"uller spaces is the following:
\begin{theorem}
If $\riem^P$ is a punctured surface of type $(g,n)$ with $2g-2+n > 0$, then $T(\riem^P)$ is a $3g-3+n$-dimensional complex manifold.
In the case of a bordered surface $\riem^B$, $T(\riem^B)$ is an infinite-dimensional manifold with
complex structure modeled on a complex Banach space.
\end{theorem}
In fact, the latter claim holds for the Teichm\"uller space of any surface covered
by the disk, by a
theorem of Bers \cite{bers-65}.  The more difficult border case follows from
this general theorem.

There are two main constructions of the complex structure on Teichm\"uller
 space \cite{Lehto,Nagbook}.  One is through the Bers embedding of the Teichm\"uller space (by Schwarzian derivatives) into a space of invariant quadratic differentials on the disk.  This is a global map onto an open subset
 of a Banach space.  The other is by proving the existence of local holomorphic sections of the
 fundamental projection $\Phi$ (e.g. using the Ahlfors-Weill or Douady-Earle reflection), and thus a complex
 structure is obtained from the space of Beltrami differentials.  This is much
 more delicate, since in some sense it must be shown that the sections are transverse
 to the Teichm\"uller equivalence relation; it is closely related to the theory of
 so-called {\it harmonic Beltrami differentials} (which model the tangent space at
 any point).  The two complex structures are equivalent.

The following two facts regarding this complex structure are essential for our
 purposes.
 \begin{theorem}
 The fundamental projection $\Phi_{\Sigma}:L^\infty_{(-1,1)}
 (\Sigma)_1 \rightarrow T(\riem)$ is holomorphic and possesses local holomorphic sections through every point.
 \end{theorem}

 An important special case is the Teichm\"uller space of the disk.
 \begin{definition}
  The \emph{universal Teichm\"uller space} is the space $T(\mathbb{D^+})$.
 \end{definition}
  The disk $\mathbb{D}^+$ is often replaced with $\mathbb{D}^-$ or the upper half plane.
  Nearly all Riemann surfaces are given by a quotient of the disk by a Fuchsian
  group.  It is possible to model Teichm\"uller spaces of such surfaces as spaces
  of invariant Beltrami differentials or quadratic differentials.  The universal
  Teichm\"uller space gets its name from the fact that
  it contains as open subsets the Teichm\"{u}ller spaces of all Riemann surfaces covered by $\disk^+$.
  We mention in passing that it has been suggested that the universal Teichm\"uller space might serve as a basis for a non-perturbative formulation of closed bosonic string theory  (see \cite{HongRajeev} and \cite{Pekonen} for an overview and references).

   We will see ahead that the
  Teichm\"uller space and the rigged moduli space are nearly the same.  This ``conformal field
  theoretic'' view of Teichm\"uller theory is in some sense very different from the Fuchsian
  group picture; however, classical Teichm\"uller theory, together with our results,
  says that they are two models of the
  same space.  It is therefore of great interest to investigate the relation between the
  algebraic structures obtained from the two pictures.  To our knowledge there are no results
  concerning this relation in the literature.

\end{subsection}


\begin{subsection}{Riemann moduli spaces}
 The \emph{Riemann moduli space} $\mathcal{M}(\riem)$ of a Riemann surface $\riem$ is the space of conformal equivalence classes of Riemann surfaces with the same topological type as $\riem$.
To understand the complex structure on the moduli space, it is necessary to study its universal cover which is the \teich space $T(\riem)$.
The \teich modular group $\tmod(\riem)$ acts holomorphically, and properly discontinuously on $T(\riem)$. In the case of punctured surfaces the modular group is the mapping class group which is finitely generated by Dehn twists. The quotient $T(\riem) / \tmod(\riem)$ is the moduli space, and this exhibits $\mathcal{M}(\riem)$ as an orbifold with complex structure inherited from its universal cover $T(\riem)$.

\end{subsection}


\end{section}

\begin{section}{Teichm\"uller space/rigged moduli space correspondence}
\label{se:Teich_RMS_correspondence}
 If one assumes that the riggings are in the class of quasiconformally extendible mappings / quasisymmetries, then Radnell and Schippers showed in \cite{RadnellSchippers_monster} that the quasiconformal
 Teichm\"uller space is, up to a discrete group action, the rigged moduli space.
 In this section we explain this connection, and some of its consequences for conformal
 field theory.  We also define the Weil-Petersson class rigged moduli and Teichm\"uller
 spaces.
\begin{subsection}{Spaces of riggings and their complex structures}
We describe three classes of riggings in both the border and puncture model.

\begin{subsubsection}{Analytic Riggings}
 We say that $\psi:\mathbb{S}^1 \rightarrow \mathbb{S}^1$ is an analytic diffeomorphism
 if it is the restriction to $\mathbb{S}^1$ of a biholomorphism of an open neighbourhood
 of $\mathbb{S}^1$ to an open neighbourhood of $\mathbb{S}^1$.  Denote the set of
 such maps by $\mathrm{A}(\mathbb{S}^1)$.

  The corresponding
 set of riggings in the puncture model is the set $A(\disk)$ of one-to-one holomorphic
 maps from $\disk$ to $\mathbb{C}$ taking $0$ to $0$ and with a holomorphic one-to-one
 extension to a disk of radius $r >1$.

For a punctured Riemann surface of type $(g,n)$, the corresponding space of riggings
$\mathcal{R}(A(\disk),\riem^P)$ can be given the structure of a complex (LB)-space, which is an inductive limit of Banach spaces \cite[Appendix B]{Huang}. The corresponding moduli space can be given a complex structure \cite{Radnell_thesis}, but this infinite-dimensional structure can not be related to the \teich space of bordered surfaces and this limits the connections between CFT and geometric function theory in general.
\end{subsubsection}

\begin{subsubsection}{Quasisymmetric and quasiconformally-extendible riggings} In this section,
we consider the class of riggings $\mathcal{R}(\qs(\mathbb{S}^1),\riem^B)$ for a
bordered surface $\riem^B$, and the corresponding puncture model.  We now describe the
equivalent class of riggings in the puncture model.
\begin{definition}
Let $\Oqc$ be the set of injective holomorphic maps $\phi : \disk \to \mathbb{C}$ such
that $\phi(0) = 0$ and $\phi$ extends quasiconformally to an open neighborhood of  $\cdisk$ (and thus to all of $\mathbb{C}$).
\end{definition}

The set $\Oqc$ can be identified with an open subset of a Banach space.  Define
\[  A_1^\infty(\mathbb{D}) = \{ \phi(z) \colon \mathbb{D} \rightarrow
 \mathbb{C} \;|\; \phi \text{ holomorphic}, \ ||\phi||_{1,\infty}=
 \sup_{z \in \mathbb{D}}
 (1- |z|^2) | \phi(z)| < \infty \}.  \]
This is a Banach space with the norm given in the definition.
Following \cite{RSnonoverlapping} based on results of L.-P. Teo \cite{Teo_Velling} we now define an embedding of $\Oqc$ into a Banach
space:
\begin{align*}
 \chi \colon \Oqc & \longrightarrow  A_1^\infty(\mathbb{D})
 \oplus \mathbb{C}  \\
 f & \longmapsto  \left( \frac{f''}{f'}, f'(0) \right).
\end{align*}
The Banach space direct-sum norm on
$A_1^\infty(\mathbb{D})\oplus \mathbb{C}$ is defined by
$
||(\phi, c)||=||\phi||_{1,\infty} +
|c|.$
\begin{theorem}[{\cite[Theorem 3.1]{RSnonoverlapping}}] \label{th:chi_induces_complex_structure}
 $\chi$ is one-to-one, and the image is open. In particular, $A_1^\infty(\disk) \oplus \mathbb{C}$
 induces a complex Banach manifold structure on $\Oqc$ via $\chi$.
\end{theorem}

 Two of the authors showed that $\Oqc$-type riggings forms a Banach manifold.
 \begin{theorem}[\cite{RSnonoverlapping}] Let $\riem^P$ be a punctured Riemann surface of type $(g,n)$.
 $\mathcal{R}(\Oqc,\riem^P)$ is a complex Banach manifold locally
 modeled on $\Oqc$.
 \end{theorem}

 The following generalization of Theorem \ref{basicqcextension} for boundary curves of a Riemann
 surface holds.    It is based on the conformal welding theorem (Theorem \ref{th:welding}).
 Using Definition \ref{de:riggings_border_model} and Remark \ref{de:riggings_border_model} the notion of a quasisymmetric map from $\mathbb{S}^1$ to $\partial_i \Sigma^B$ is well defined.
 \begin{theorem}[\cite{RadnellSchippers_monster}] \label{rigext}
  Let $\Sigma^B$ be a bordered Riemann surface, with boundary curve $\partial_i \Sigma$
  homeomorphic to $\mathbb{S}^1$.   A map $h:\mathbb{S}^1 \rightarrow \partial_i \Sigma^B$
  is a quasisymmetry if and only if $h$ has a quasiconformal
  extension to a map from an annulus $\mathbb{A}_1^r$ to a collar of $\partial_i \Sigma^B$.
 \end{theorem}
 From this it can be shown that the puncture model
 and border model of the rigged moduli spaces are bijective
 as sets (Theorem \ref{th:modulibijection}).  In fact, they are biholomorphic.
\end{subsubsection}

\begin{subsubsection}{$\mathrm{WP}$-class quasisymmetries} We first define this class of riggings in the puncture picture.
The Weil-Petersson class quasisymmetries of the circle were introduced independently
by H. Guo \cite{GuoHui}, G. Cui \cite{Cui} and L. Takhtajan and L.-P. Teo \cite{Takhtajan_Teo_Memoirs}.
In brief, they arise from an attempt to model the Teichm\"uller space on $L^2$ Beltrami
differentials (note that the differentials must be both $L^2$ and $L^\infty$ bounded).
More background is
given in Section \ref{se:WPTeichspace} ahead.

The Bergman space is the Hilbert space
\[  \aonetwo = \left\{ \phi \in \mathcal{H}(\mathbb{D}) : \| \phi\|_2^2 = \iint_{\mathbb{D}}
  |\phi|^2 \, dA < \infty \right\} \]
which is a vector subspace of $\aoneinfinity$. The inclusion map from $\aonetwo$
to $\aoneinfinity$ is bounded \cite[Chapter II Lemma 1.3]{Takhtajan_Teo_Memoirs}.

\begin{definition}  Let
  \[ \Oqco= \left\{f \in \Oqc : f''/f' \in \aonetwo \right\}.  \]
	We call elements of this space WP-class quasiconformally extendible maps of $\disk$.
\end{definition}

\begin{theorem}{\cite[Theorem 2.3]{RSS_Filbert1}}  \label{th:Oqco_open_in_Oqc} The inclusion map from
$\aonetwo \rightarrow \aoneinfinity$ is continuous.  Furthermore
$\chi(\Oqco)$ is an open subset of the Hilbert space $\aonetwo
\oplus \mathbb{C}$ and thus $\Oqco$ inherits a Hilbert manifold structure.
The inclusion map $\iota: \Oqco \rightarrow \Oqc$ is holomorphic.
\end{theorem}
\begin{remark} Although the inclusion map is continuous, the topology of $\Oqco$ is not
the relative topology inherited from $\Oqc$.
\end{remark}

We now define the space of corresponding riggings in the border picture which we denote by $\qso(\mathbb{S}^1)$. Recall Section \ref{se:notation} and Definition \ref{quasisymmetriccircle}. Briefly, a map $h:\mathbb{S}^1 \to \mathbb{S}^1$ is in $\qso(\mathbb{S}^1)$ if the corresponding welding maps (see Theorem \ref{th:welding}) are in $\Oqco$. For $h \in \qs(\mathbb{S}^1)$ let $w_{\mu}(h): \disk^- \to \disk^-$  be a quasiconformal extension of $h$ with dilatation $\mu$ (such an extension exists by the Ahlfors-Beurling extension theorem). Furthermore, let $w^{\mu} : \sphere \to \sphere$
 be the quasiconformal map with dilatation $\mu$ on $\disk^-$
 and $0$ on $\disk^+$, with normalization $w^\mu(0)=0$, ${w^\mu}'(0)=1$ and $w^\mu(\infty)=\infty$ and set
 \[  F(h)=\left. w^\mu \right|_{\disk^+}.  \]  It is a standard fact that $F(h)$ is independent of the choice of extension $w_\mu$.
\begin{definition}
 We define the subset $\qso(\mathbb{S}^1)$ of $\qs(\mathbb{S}^1)$ by
$$
\qso(\mathbb{S}^1) = \{ h \in \qs(\mathbb{S}^1) \,:\,  F(h)  \in \Oqco \},
$$
and call elements of $\qso(\mathbb{S}^1)$ Weil-Petersson (or WP)-class quasisymmetries
\end{definition}

The WP-class quasisymmetries can also be characterized in terms of the existence of
quasiconformal extensions with $L^2$ Beltrami differentials with respect to the
hyperbolic metric.
\begin{theorem}[\cite{Cui,GuoHui, Takhtajan_Teo_Memoirs}]  Let $\phi \in \qs(\mathbb{S}^1)$.
 The following statements are equivalent:
 $(1)$ $\phi \in \qso(\mathbb{S}^1)$; $(2)$
 $\phi$ has a quasiconformal extension to $\disk^-$ $($or $\disk^+$$)$ with Beltrami differential
 $\mu$ satisfying
 \begin{equation} \label{eq:L2hyp_condition}
  \iint_{\mathbb{D}^-} \frac{|\mu(\zeta)|^2}{(1-|\zeta|^2)^2} dA_\zeta < \infty;
 \end{equation} $(3)$ $F(\phi)$ has a quasiconformal extension with Beltrami differential
 satisfying $(\ref{eq:L2hyp_condition})$.
\end{theorem}

Finally, we mention the following result which seems to be proven for the first time by G. Cui \cite{Cui}.
\begin{theorem}
 \label{th:WP_topological_group}
 The $\mathrm{WP}$-class quasisymmetries $\qs_{\mathrm{WP}}(\mathbb{S}^1)$ are a topological group under composition.
\end{theorem}
\begin{remark}  \label{re:quasisymmetries_not_group}
 The quasisymmetries do not form a topological group; composition on the left
 is continuous but composition on the right is not \cite{Lehto}.
\end{remark}

The authors showed that many of these results extend to bordered surfaces of type $(g,n)$.
One such result is the following.
\begin{theorem} \cite{RSS_Filbert1} Let $\riem^P$ be a punctured Riemann surface of type $(g,n)$.
 $\mathcal{R}(\Oqc_{\mathrm{WP}},\riem^P)$ is a Hilbert manifold locally modelled on
 $\Oqc_{\mathrm{WP}}$.
\end{theorem}

  The chain of containments for the riggings is as follows:
  \[  \mathrm{A}(\mathbb{S}^1) \subsetneq \mathrm{QS}_{\mathrm{WP}}(\mathbb{S}^1)
  \subsetneq \mathrm{QS}(\mathbb{S}^1) \subsetneq
      \mathrm{Homeo}(\mathbb{S}^1)  \]
 which implies that
 \[  \mathcal{R}(\mathrm{A}(\mathbb{S}^1),\riem^B) \subsetneq \mathcal{R}(\mathrm{QS}_{\mathrm{WP}}(\mathbb{S}^1),\riem^B)
  \subsetneq \mathcal{R}(\mathrm{QS}(\mathbb{S}^1),\riem^B)  \]
  and similarly for the puncture model of rigged moduli space.
\end{subsubsection}

\end{subsection}
\begin{subsection}{The CFT rigged moduli space is \teich space}  \label{se:sub_correspondence}
Two of the authors \cite{RadnellSchippers_monster} showed that the rigged moduli space $\widetilde{M}^B(\riem^B)$ introduced by CFT is in fact  equal to the classical $T^B(\riem^B)$ quotiented by a holomorphic discrete group action. We strongly emphasize that this is only true when quasisymmetric parametrizations are used.  There is no Teichm\"uller space of bordered
surfaces based on diffeomorphisms or analytic parametrizations in the literature.

The basic idea is the following. Let $\riem^B$ be the fixed bordered base surface for the \teich space $T^B(\riem^B)$. Now fix a base parametrization $\tau = (\tau_1, \ldots, \tau_n)$, where $ \tau_i: \mathbb{S}^1 \to \partial_i  \riem$ are quasisymmetric. Take any $[\riem^B, f, \riem^B_1] \in T^B(\riem^B)$. The composition  $\tau \circ f|_{\partial \riem}$ is a parametrization of $\riem^B_1$. Because $f$ is quasiconformal its boundary values are quasisymmetric and
so $f \circ \tau|_{\mathbb{S}^1}$ is again quasisymmetric.

In order to forget the data of the marking maps in \teich space we need to quotient by a certain mapping class group. Define $\mathrm{PModI}(\riem^B)$ to be the group of equivalence classes
of quasiconformal maps from $\riem^B$ which are the identity on $\partial \riem^B$.  Two
such quasiconformal maps are equivalent if they are homotopic via a homotopy which is
constant on the boundary.  The mapping class group acts on the Teichm\"uller space via
composition: for $[\rho] \in \mathrm{PModI}(\riem^B)$,
$[\rho][\riem^B,f,\riem_1^B] = [\riem^B,f \circ \rho^{-1},\riem_1^B]$.

After taking the quotient by $\mathrm{PModI}(\riem^B)$,
the information of the marking map on the boundary which remains is precisely the data of the parametrization.
\begin{theorem}{\cite[Theorems 5.2, 5.3]{RadnellSchippers_monster}}  \label{th:correspondence}
The mapping class group $\mathrm{PModI}(\riem^B)$ acts properly discontinuously and fixed-point freely by biholomorphisms on $T^B(\riem^B)$.
The quotient is $\widetilde{M}^B(\riem^B)$, and so this rigged moduli space inherits the structure of an infinite-dimensional complex Banach manifold from $T^B(\riem^B)$.
\end{theorem}

It will be shown ahead that the puncture model of rigged moduli space is bijective
 to the border model, (Theorem \ref{th:modulibijection}), so it also inherits a complex Banach
 manifold structure. Finally, we observe the following consequence of the Theorem \ref{th:correspondence} for Teichm\"uller theory.
 \begin{theorem}[\cite{RadnellSchippers_fiber}] \label{th:fiber_structure}
  Let $\riem^B$ be a bordered Riemann surface of type $(g,n)$ and let
  $\riem^P$ be a punctured surface of type $(g,n)$ obtained by sewing on $n$ punctured
  disks.  The Teichm\"uller space $T(\riem^B)$ is a holomorphic fiber space over the (finite-dimensional) base $T(\riem^P)$.  The fibre over any point $[\riem^P,f_1,\riem_1^P]
   \in T(\riem^P)$ is $\mathcal{R}(\Oqc,\riem_1^P)$.  Furthermore, the complex structure
   of $\mathcal{R}(\Oqc,\riem_1^P)$ agrees with the restriction of the complex structure of the
   total space $T(\riem^B)$.
 \end{theorem}
 We will describe the operation of sewing on caps in detail in Section \ref{se:sewingoperation} ahead.
 Theorem \ref{th:fiber_structure} is a consequence of the correspondence and the
 conformal field theoretic idea of rigging.  The holomorphicity of the fibration and its
 agreement with the complex structure of $\mathcal{R}(\Oqc,\riem_1^P)$ is non-trivial
 and is based on a method of constructing coordinates on the rigged moduli
 space via Gardiner-Schiffer variation in Radnell's thesis \cite{Radnell_thesis}.
\end{subsection}
\begin{subsection}{Weil-Petersson class Teichm\"uller space}  \label{se:WPTeichspace}
 On finite-dimensional Teichm\"uller spaces (those of finite genus
 punctured Riemann surfaces), the so-called Weil-Petersson pairing
 of Beltrami differentials defines a Riemannian metric which is automatically well-defined.
 On infinite-dimensional Teichm\"uller spaces, this pairing is not finite in general, as
 was observed by S. Nag and A. Verjovsky \cite{NagandVerjovsky}. G. Cui \cite{Cui}, H. Guo \cite{GuoHui},
 and L. Takhtajan and L.-P. Teo \cite{Takhtajan_Teo_Memoirs} independently defined a subset of the universal Teichm\"uller space
 on which the Weil-Petersson pairing is finite.  Furthermore, it is a Hilbert manifold.
 We will not discuss the Weil-Petersson metric in this paper, except
 to mention in passing that it is of significant interest in Teichm\"uller theory.

 The Weil-Petersson Teichm\"uller space has attracted a great deal of attention
 \cite{Fi, GR, HS, KU, TT, Takhtajan_Teo_Memoirs, Wu}. Up until recently, the only example
 of Teichm\"uller spaces with convergent Weil-Petersson pairing, aside
 from the finite dimensional Teichm\"uller spaces, was the Weil-Petersson universal Teichm\"uller
 space.  In a series of papers \cite{RSS_Filbert2,RSS_Filbert1,RSS_WP1,RSS_WP2}
 the authors defined a Weil-Petersson Teichm\"uller
 space of bordered surfaces of type $(g,n)$, based on the fiber structure on Teichm\"uller
 space derived from the rigged moduli space (Theorem \ref{th:fiber_structure}).
 M. Yanagishita independently gave a definition which includes
 these surfaces, based on the Bers embedding of $L^2 \cap L^\infty$
  Beltrami differentials into an open subset of the quadratic differentials, using the Fuchsian group point of view \cite{Yan}.  By results of
  \cite{RSS_WP1}, these must be the same space.
 These complex structures are also very likely equivalent, as they are in the classical $L^\infty$
  case by Theorem \ref{th:fiber_structure}, but this has not yet been demonstrated.
  We also described the tangent space in terms of infinitesimal Beltrami differentials
  \cite{RSS_WP2}.

 We now define the Weil-Petersson class Teichm\"uller space.
 \begin{definition}
  Let $\riem^B$ be a bordered Riemann surface of type $(g,n)$.  The Weil-Petersson
  class Teichm\"uller space is
  \[  T_{\text{WP}}(\riem^B) = \{ \riem^B,f,\riem_1^B \}/\sim  \]
  where $f:\riem^B \rightarrow \riem_1^B$ is a quasiconformal map whose
  boundary values are WP-class quasisymmetries, and two elements are
  equivalent $(\riem^B,f_1,\riem^B_1) \sim (\riem^B,f_2,\riem^B_2)$ if and only
  if there is a biholomorphism $\sigma:\riem^B_1 \rightarrow \riem^B_2$ such
  that $f_2^{-1} \circ \sigma \circ f_1$ is homotopic to the identity rel boundary.
 \end{definition}
 Weil-Petersson class quasisymmetries between borders of Riemann
 surfaces are defined using collar charts and ideal boundaries.  For details
 see \cite{RSS_Filbert2}.  We have
 \begin{theorem}[\cite{RSS_WP1}] Let $\riem^B$ and $\riem^B_1$ be
  bordered Riemann surfaces of type $(g,n)$.
  Let $f:\riem^B \rightarrow \riem_1^B$ be a quasiconformal map whose
  boundary values are $\mathrm{WP}$-class quasisymmetries.  There is a quasiconformal
   map $F:\riem^B \rightarrow \riem_1^B$
  which is homotopic to $f$ rel boundary and has hyperbolically $L^2$ Beltrami
  differential in the sense of equation $(\ref{eq:L2hyp_condition}).$

  In particular, any $[\riem,f_1,\riem_1^B] \in T_{\mathrm{WP}}(\riem^B)$ has a
  representative with hyperbolically $L^2$ Beltrami differential.
 \end{theorem}
 Thus our definition of the Weil-Petersson class Teichm\"uller space agrees
 with the definition of Yanagishita \cite{Yan}.

 \begin{theorem}[\cite{RSS_Filbert2}]  \label{th:WP_Hilbert_manifold}
 Let $\riem^B$ be a bordered surface of type $(g,n)$.
  The Weil-Petersson class Teichm\"uller space is a complex Hilbert manifold.
 \end{theorem}
 \begin{remark} We also showed that the tangent space is modelled by so-called
  $L^2$ harmonic Beltrami differentials, and therefore the Weil-Petersson pairing
  of Beltrami differentials (which defines a Hermitian metric) is finite \cite{RSS_WP2}.
 \end{remark}
 The space $T_{\mathrm{WP}}(\riem^B)$ is a cover of the rigged moduli space.  We define the mapping class group action on
 $T_{\text{WP}}(\riem^B)$ as in Section \ref{se:sub_correspondence}, but
 restricting to WP-class quasiconformal maps.
 \begin{theorem}[\cite{RSS_Filbert2}]  \label{th:WP_mapping_class_prop}
 The mapping class groups acts properly discontinuously and fixed-point freely by biholomorphisms on $T_{\mathrm{WP}}(\riem^B)$.
 The quotient is $\widetilde{M}^B(\riem^B)$, and so the rigged moduli space $\mathcal{R}(\qs_{\mathrm{WP}},\riem^B)$, inherits the structure of an infinite-dimensional complex Hilbert manifold from $T_{\mathrm{WP}}(\riem^B)$.  The same holds for $\mathcal{R}(\Oqco,\riem^P)$.
 \end{theorem}
\end{subsection}
\begin{subsection}{The sewing operation}\label{sewingoperation}
\label{se:sewingoperation}

 The central geometric  operation in CFT is the \emph{sewing operation} which is materialised when two rigged Riemann surfaces are joined along a boundary curve by using the riggings.
 The algebraic structure of CFT is encoded geometrically by the sewing operation as a map between rigged moduli spaces.
Define $J : \mathbb{S}^1 \to  \mathbb{S}^1$ by $J(z) = 1/z$.

 \begin{definition}
The sewing operation between the rigged Riemann surfaces $(\Sigma^B_1,\psi^1)$ and $(\Sigma^B_2,\psi^2)$  is defined as follows: For boundary curves $\partial_i \Sigma^B_1$ and
 $\partial_j \Sigma^B_2$, for fixed $i$ and $j$, define
 $  \Sigma^B_1 \#_{ij} \Sigma^B_2 = (\Sigma^B_1 \sqcup
 \Sigma^B_2)/\sim  \,$,
 where two boundary points $p_1 \in \partial_i \Sigma^B_1$ and
 $p_2 \in \partial_j \Sigma^B_2$ are equivalent if and only if
 $p_2=(\psi^2_j \circ J \circ (\psi^1_i)^{-1})(p_1)$. The role of the reciprocal is to produce an orientation reversing map.
 \end{definition}
 There is a natural way to make $\Sigma^B_1 \#_{ij} \Sigma^B_2$ into a
topological space. If $\psi^1_i$ and $\psi^2_j$ are analytic
parametrizations then $\Sigma^B_1 \#_{ij} \Sigma^B_2$ becomes a Riemann surface
in a standard way using $\psi_1$ and $\psi_2$ to produce charts on the
seam. See for example L. Ahlfors and D. Sario \cite[Section II.3D]{Ahlfors_Sario}.
However this can be done in much greater generality using quasisymmetric parametrizations.

Conformal welding (Theorem \ref{th:welding}) was previously used by two of the authors to sew Riemann surfaces using quasisymmetric boundary identifications.
\begin{theorem}{\cite[Section 3]{RadnellSchippers_monster}}
If $\psi^1$ and $\psi^2$ are quasisymmetric riggings then $\Sigma^B_1 \#_{ij} \Sigma^B_2$ has a unique complex structure such that the inclusions of $\Sigma^B_1$ and $\Sigma^B_2$ are holomorphic.
\end{theorem}

A useful special case is the sewing of caps onto a bordered surface $\riem^B$ of type $(g,n)$  to obtain a punctured surface of type $(g,n)$.
The punctured disk $\overline{\disk}_0 = \{z \in \Bbb{C} \st 0 <
|z| \leq 1 \}$ will be considered as a bordered Riemann surface
whose boundary is parametrized by $z \mapsto 1/z$.

Let $\tau=(\tau_1,\ldots, \tau_{n})$ be a quasisymmetric rigging of $\riem^B$. At each boundary curve $\partial_i \riem^B$ we sew in the punctured
disk $\overline{\disk}_0$ using $\tau_i$ as
described above.  We denote the simultaneous sewing by $\riem^B\#_\tau (\cdisk_0)^{n}$
and let $\riem^P=\riem^B\#_\tau (\cdisk_0)^{n}$ be the resultant
punctured surface. The images of the punctured disks in
$\riem^P$ will be called \textit{caps}.
We can remember the rigging data by defining a homeomorphism  $\tilde{\tau}: \overline{\disk}_0 \to  \riem^P$ by
$$\tilde{\tau}(z) =
\begin{cases}
z, & \quad z \in \disk_0 \\
\tau(z), & \quad z \in \mathbb{S}^1 \\
\end{cases}
$$
Note that $\tilde{\tau}$ is holomorphic on $\disk$ and thus defines a rigging in the puncture model.
\begin{remark} \label{re:riggings_switcheroo}
 Conversely, given a puncture type rigging $f_i:\disk \rightarrow \riem^P$, $\left. f_i \right|_{\mathbb{S}^1}$ is a border-type rigging of $\riem^P \setminus$caps.
\end{remark}

\begin{definition} \label{de:E_puncture_border_relation}
 For any class of parametrization for which sewing is defined we let $\mathcal{E} :\widetilde{\mathcal{M}}^B(\riem^B) \to \widetilde{\mathcal{M}}^P(\riem^P)$ be given by $\mathcal{E}(\riem^B, \tau) = (\riem^P, \tilde{\tau})$.
The inverse of this map is defined by cutting out the image of the disk using the rigging and remembering the rigging on the boundary $\mathbb{S}^1$.
\end{definition}
It is straightforward to check that $\mathcal{E}$ and its inverse are well defined maps on moduli space.

After sewing using a quasisymmetric parametrization,
the seam in the new punctured surface is a highly irregular curve known as a quasicircle.  Curves
 usually called ``fractal'' are quasicircles.  For the
purpose of this paper, we define a quasicircle as follows.
\begin{definition}
 A quasidisk in a compact Riemann surface $\riem$ is an open connected subset $\Omega$ of $\riem$
 such that there is a local biholomorphic coordinate $\zeta:U \rightarrow \mathbb{C}$
 so that $U$ contains the closure of $\Omega$ and $\zeta(\Omega)=f(\disk)$ for some
 $f \in \Oqc$.
 A quasicircle is the boundary of a quasidisk.  If the corresponding $f$ is in $\Oqc_{\mathrm{WP}}$,
 we refer to it as a WP-class quasicircle.
\end{definition}
In punctured Riemann surfaces we do not alter the definition and apply it as
though the puncture is filled in (if necessary one can refer to punctured quasidisks
with the obvious meaning).
The complement of a quasidisk in the Riemann sphere is also a quasidisk.
In the case of Weil-Petersson class riggings, the boundary curve is more regular, in fact
rectifiable \cite{RSS_WPjump}.

\begin{theorem}
\label{th:modulibijection}
The map $\mathcal{E}$ defines bijections between the border and puncture models of the moduli space for each the three different classes of parametrizations, analytic, quasisymmetric and $\mathrm{WP}$-class quasisymmetric.
\end{theorem}

It is straightforward to check that the sewing operation gives a well-defined map between moduli spaces. Working on the \teich space level, the holomorphicity of the sewing operation follows quite naturally.

\begin{theorem}[\cite{RadnellSchippers_monster}]  \label{th:sewing_holomorphic}
 The sewing operation on the quasisymmetrically rigged moduli space
 of bordered surfaces is holomorphic.  Similarly, the sewing operation on $\Oqc$-rigged
 punctured moduli space is holomorphic.
\end{theorem}

Work in progress by the authors indicates
that sewing is a holomorphic operation in Weil-Petersson class Teichm\"uller space.
\end{subsection}

\end{section}
\begin{section}{Analytic setting for the determinant line bundle}
 \label{se:analytic_setting_detline}

The central charge (or conformal anomaly) plays an important role in the physics, algebra and geometry of CFT. Geometrically it is encoded by the determinant line bundle over the rigged moduli space. See \cite{Huang, SegalPublished}.

The determinant line bundle arises from decompositions of functions
 on the boundaries of the Riemann surfaces into Fourier series with
 positive and negative parts.   Typically, the determinant line bundle is
 defined using analytic riggings.  In fact, quasisymmetric mappings are exactly
 the largest class of riggings for which these decompositions are
 defined.   The determinant line bundle itself
 requires restricting to WP-class quasisymmetric riggings, and
 these mappings arise quite naturally from the definition of the determinant
 line bundle.  We will see this in the next few sections.

 In  Section \ref{se:decompositions}, we define the relevant spaces of decompositions of Fourier series. One of our main contributions is our recognition of the role of
 the Dirichlet spaces.
 In Section \ref{se:jump}, we outline some of our results on the jump
 formula in the setting of quasicircles and WP-class quasicircles.
 This is a key tool for investigating the decompositions in the puncture setting.   In
 Section \ref{se:SW_Grunsky}, we describe the Segal-Wilson universal Grassmannian, which
 by results of S. Nag and D. Sullivan \cite{NagSullivan} can be seen as a generalization of the
 classical period map for compact surfaces to the case of the disk.  We also
 describe the relation of the projections described in Section \ref{se:decompositions}
 to the Grunsky operator, a classical construction in geometric function theory.
 We show that the quasisymmetries are the largest class of riggings preserving
 the $H^{1/2}$ space of the boundary of a Riemann surface, and that the Fredholm property of
the operator defining the determinant lines requires the use of quasisymmetric riggings and the
existence of holomorphic sections of the determinant line bundle requires the WP-class riggings.

\begin{subsection}{Decompositions of Fourier series and the determinant line bundle}
 \label{se:decompositions}

 We will decompose Fourier series of functions on $\partial_i \riem^B$ into
 positive and negative parts.  The Fourier series themselves are obtained by pulling
 back the boundary values to the circle $\mathbb{S}^1$.  We first define
 the analytic class of the boundary values.
 \begin{definition} \label{de:sobolevboundary}
 Let $\riem^B$ be a bordered Riemann surface of type $(g,n)$.
  Denote the complex Sobolev
  space $H^{1/2}(\partial_i \riem^B)$ by $\mathcal{H}(\partial_i \riem^B)$.
 \end{definition}
  This can be defined precisely and conformally invariantly using the definition
  of border \cite{Ahlfors_Sario}, or by considering the boundary as an analytic
  curve in the double of $\riem^B$.  So the statement that $h\in\mathcal{H}(\partial_i \riem^B)$ is equivalent to the condition that
  for a fixed collar chart $\zeta_i$ of each boundary $\partial_i \riem$, the
  function $h \circ \zeta_i^{-1}$ is in $H^{1/2}(\mathbb{S}^1)$ (and this
  characterization is independent of $\zeta_i$).
 \begin{remark}[On this choice of function space]
  The space $\mathcal{H}(\partial_i \riem^B)$ is the largest function
  space on the boundary for which the determinant line bundle arising from
  decompositions can be defined.
  See Huang \cite[Corollary D.3.2., Theorem D.3.3.]{Huang}.   We will return
  to this point ahead.
 \end{remark}

  If we restrict to the circle, we get a familiar special case in classical
 complex analysis.  Namely,
  $\mathcal{H}(\mathbb{S}^1)$ is the set of functions $h:\mathbb{S}^1 \rightarrow \mathbb{C}$  in $L^2 (\mathbb{S}^1 )$
 such that the Fourier series $h(z) = \sum_{n=-\infty}^\infty h_n e^{i n \theta}$
  satisfies
  \[  \sum_{n=-\infty}^\infty |n| |h_n|^2 <\infty.   \]
  The subset of real-valued elements of $\mathcal{H}(\mathbb{S}^1)$
  (equivalently, those such that $h_n = \overline{h_{-n}}$ for all $n$) is denoted
  by $\mathcal{H}_\mathbb{R}(\mathbb{S}^1)$.

  Let $\mathcal{D}(\disk^\pm)$ denote the Dirichlet spaces of $\disk^\pm$. These spaces consist of the holomorphic functions on $\disk^\pm$ with
  finite Dirichlet energy, that is,
  \[  \iint_{\disk^\pm} |h'|^2 \,dA <\infty.  \]
  We will assume that elements of $\mathcal{D}(\disk^-)$ are holomorphic at $\infty$
  and vanish there.
  It is a classical fact that one has
  \begin{align*}
     \mathcal{D}(\disk^+) & = \left\{ h \in \mathcal{H}(\mathbb{S}^1) \,:\,
      h = \sum_{n=0}^\infty h_n  e^{in\theta} \ \ \ \right\}  \\
      \mathcal{D}(\disk^-) & = \left\{ h
       \in \mathcal{H}(\mathbb{S}^1) \,:\, h=
      \sum_{n=-\infty}^{-1} h_n  e^{in\theta} \right\} \\
  \end{align*}
  by replacing $e^{i\theta}$ with $z$.
  In this section, we will use the convention that the constant term always belongs
  in $\mathcal{D}(\disk^+)$.
  Thus we have the canonical decomposition
  \begin{equation} \label{eq:HSonedecomp}
   \mathcal{H}(\mathbb{S}^1)=  \mathcal{D}(\disk^+)\oplus  \mathcal{D}(\disk^-),
  \end{equation}
  and therefore, from now on we do not comment on the restriction or extension to simplify notation.
  Let $P(\disk^\pm): \mathcal{H}(\mathbb{S}^1) \rightarrow \mathcal{D}(\disk^\pm)$
  denote the projection onto the components.

 Given a function $g$, define the composition map $\mathcal{C}_g$ by $\mathcal{C}_g(h) = h \circ g$ in any appropriate setting.  We have the following result.
 \begin{theorem}  \label{th:qs_riggings_preserve_H}
  Let $\riem^B$ be a bordered Riemann surface of type $(g,n)$ and let
  $\phi = (\phi_1,\ldots,\phi_n) \in \mathcal{R}(\qs(\mathbb{S}^1),\riem^B)$.
  For each $i$, the map $\mathcal{C}_{\phi_i}:\mathcal{H}(\partial_i \riem) \rightarrow
  \mathcal{H}(\mathbb{S}^1)$ is a bounded isomorphism.
 \end{theorem}
 \begin{proof}  This follows from the characterization in terms of collar charts
  $\zeta_i$ and \cite[Theorem 2.2]{RSS_WPjump}, together with the conformal invariance
   of Dirichlet space.
 \end{proof}
 \begin{remark}
  Theorem \ref{th:NagSullivan_bounded} ahead is almost enough in place of
  \cite[Theorem 2.2]{RSS_WPjump}, but it does not
  control the constant term; this is more delicate and requires a Poincar\'e inequality.
 \end{remark}

We examine the decompositions of the Fourier series on the boundary.
 Let $(\riem^B,(\phi_1,\ldots,\phi_n))$ be a rigged bordered surface. Using the decomposition \eqref{eq:HSonedecomp} we can make sense of Fourier series for functions on the boundary as follows:
 \begin{definition}
  For any boundary curve $\partial_i \riem^B$,  define
  \begin{equation*}
    F_\pm(\partial_i \riem^B) = \{  h \in \mathcal{H}(\partial_i \riem^B) \,:\,
     h \circ \phi_i  \in \mathcal{D}(\disk^\mp)  \}.
  \end{equation*}
 \end{definition}
 Here ``F'' stands for ``Fourier series''. In other words, we use the riggings
 to enable the definition of a Fourier series for functions on the boundary, and then
 we obtain decompositions of those Fourier series into plus or minus parts.
Recall that by construction, all our boundary parametrizations are negatively oriented, and this explains the $\mp$ in the definition of $F_{\pm}$.
In CFT it is customary to use both positively and negatively oriented boundary parametrizations, but in studying the analytic issues we can ignore this extra data and its effect on where the constants appear in $F_{\pm}$.

 From our results outlined above we have the following result.
 \begin{theorem}
  The projection operators
  \[  P(\partial_i \riem^B)_\pm:\mathcal{H}(\partial_i \riem^B) \longrightarrow F_\pm(\partial_i
  \riem^B)  \]
  are bounded and given by
  \[  P(\partial_i \riem^B)_\pm = \mathcal{C}_{\phi^{-1}_{i}} \circ P(\disk^{\mp}) \circ
    \mathcal{C}_{\phi_i}.    \]
 \end{theorem}
 The formula for the operators is more or less obvious.  However an analytic proof
 requires among other things the boundedness of the composition operator $\mathcal{C}_{\phi_i}$
 on the space $\mathcal{H}(\partial_i \riem)$ (see Theorem \ref{th:qs_riggings_preserve_H}).

 Define the Dirichlet space of holomorphic functions on $\riem^B$ to be
 \[  \mathcal{D}(\riem^B) = \left\{ h:\riem^B \rightarrow \mathbb{C} \,:\, h \ \ \ \text{holomorphic}
   \ \text{and} \ \ \iint_{\riem^B} dh \wedge \overline{dh} <\infty \right\}.  \]
 The determinant line bundle will be defined using the operator
\begin{align*}
\pi : \mathcal{D}(\riem^B) & \longrightarrow \bigoplus_{i=1}^n F_+(\partial_i \riem^B) \\
g & \longmapsto  \left( P(\partial_1 \riem^B)_+  (g|_{\partial_1 \riem^B}), \ldots,   P(\partial_n \riem^B)_+  (g|_{\partial_n \riem^B}) \right)
\end{align*}
 Next we define the determinant line (see \cite[Appendix D]{Huang} for details).
 \begin{definition} \label{detline}
  The determinant line associated to the rigged Riemann surface $(\riem^B, \phi)$ is the line
  \[  \text{Det}(\pi) = \text{Det}(\text{ker}(\pi))^* \otimes \text{Det}(\text{coker}(\pi)).  \]
	where Det of a $k$-dimensional vector space is its $k$th exterior power.  These lines should be $\mathbb{Z}_2$ graded but this is not important for our current discussions.
 \end{definition}
 \begin{remark}  In \cite[Appendix D]{Huang_CFT}, Huang defines the operator $\pi$ on
  holomorphic functions with smooth extensions to $\partial \riem^B$, and then in Proposition
  D.3.3. extends it to the completion of these functions in the Sobolev space $H^s(\riem^B)$ for all $s \geq 1$.
  It can be shown, that for $s=1$ this is precisely $\mathcal{D}(\riem^B)$.  Thus we have
  a very satisfying connection with classical function theory.  Note also that
  $\mathcal{D}(\riem^B)$ is manifestly conformally invariant.

  The boundary values of
  elements of $\mathcal{D}(\riem^B)$ are
  in $H^{1/2}(\partial_i \riem^B)$ for all $i$.  In fact, up to topological obstructions,
  $H^{1/2}(\partial \riem^B)$ is exactly the boundary values of the complex harmonic Dirichlet space.

 \end{remark}

 Definition \ref{detline} requires that $\pi$ is Fredholm, which holds for quasisymmetric riggings. See Theorem \ref{Fiber identity} for the case of genus zero and one boundary curve. For analytic boundary parametrizations this is
 shown in \cite[Theorem D.3.3]{Huang_CFT}.

In the analytic case, the determinant lines form a holomorphic line bundle.
\begin{theorem}
The determinant lines form a holomorphic line bundle over the rigged moduli space $\widetilde{\mathcal{M}}^B(A(\mathbb{S}^1))$.
\end{theorem}
Previously this has only been studied in the case of analytic riggings.
In genus-zero this was proved first by Huang \cite{Huang} and in higher-genus by Radnell \cite{Radnell_thesis}.
As well as the issues pointed out in Section \ref{ss:CFT}, how to define the determinant line over a rigged moduli space element is non-trivial.
In genus-zero (and one), there are obvious canonical representatives of each conformal equivalence class of surface, but in higher-genus the universal \teich curve must be used, and the det line bundle should be thought of as lying over this space. The case of WP-class riggings is under consideration by the authors.  We conjecture that the determinant line bundle is a holomorphic line bundle
in this case.

\end{subsection}
\begin{subsection}{The Plemelj-Sokhotski jump formula}  \label{se:jump}
 The decompositions in the previous section are closely related to the Plemelj-Sokhotski
 jump formula.  Furthermore, Huang \cite{Huang} showed that the cokernel of $\pi$ can be
 computed in genus zero through the use of the jump formula (note that this is interesting
 only if one carries the constants through).
 Radnell \cite{Radnell_thesis} showed
 that these techniques extend to higher genus, with explicit computations in genus one.

  In a sense,
 extending to higher genus is a topological rather than an analytic problem - the analytic issues arise already in
 genus zero.  The remainder of this section will therefore only deal with the genus zero case.
 We will touch on the issue of higher genus briefly in Section \ref{se:SW_Grunsky}.

 Roughly, the jump formula says that for a reasonable function $h$ on a reasonably regular Jordan curve $\Gamma$, if one forms the Cauchy integral
 \[  h_{\pm}(z) = \frac{1}{2\pi i} \int_{\Gamma} \frac{h(\zeta)}{z- \zeta} d\zeta, \quad \text{for } z \in \Omega^\pm,   \]
 where $\Omega^+/ \Omega^-$ are the bounded/unbounded components of the complement of $\Gamma$, and
 $h_{\pm}$ has domain  $\Omega^{\pm}$,
 then the resulting holomorphic functions $h_{\pm}$ satisfy
 \[  \left. (h_{+} - h_{-}) \right|_{\Gamma} = h.  \]

 For this to make sense, the regularity of the curve and of $h$ must be such that
 there is a sensible notion of boundary values of the holomorphic functions, or
 at least of the decomposition.
 It is also desirable that the Cauchy integral is a bounded operator.  Thus there are
  three interrelated analytic choices to be made: (1) the class of functions on the curve;
 (2) the class of holomorphic functions on the domains; and (3) the regularity
 of the curve.  We have chosen the class of the functions on the curve to be
 in the Sobolev $H^{1/2}$ space based on results of Y.-Z. Huang \cite{Huang}.  The regularity of the
 curve is determined by the class of border-type riggings.  In this section and
 the next, we will show that with this choice, the natural class of functions is the
 Dirichlet space and the riggings are quasisymmetries.  On the other hand,
 the perfect correspondence between the three choices can be seen as further evidence of the
 naturality of Huang's choice.

 If one sews on caps using quasisymmetric riggings, then the resulting seam is a quasicircle
 in the Riemann surface (in this section, a sphere). Thus, we must extend the formula to quasicircles, or at the very least to
 WP-class quasicircles.
 Schippers and Staubach extended the Plemelj-Sokhotski decomposition to quasicircles in the sphere \cite{SchippersStaubach_Osborn},
 using a limiting integral (however the jump in terms of the limiting values though likely
 has not been proven).
 In the case of WP-class quasicircles, we showed that the curve is rectifiable and the jump
 formula holds for the ordinary Cauchy integral \cite{RSS_WPjump}.
  We outline these results here.

Let $\Gamma$ be a quasicircle in the sphere not containing $0$
   or $\infty$, and let $\Omega^+$ and $\Omega^-$ be the bounded and unbounded components of the complement respectively.

  \begin{definition} \label{de:harmDirichlet}
  The harmonic Dirichlet space $\mathcal{D}(\Omega^\pm)_{\mathrm{harm}}$ is the set of all complex-valued
  harmonic functions $h$ on $\Omega^\pm$ such that
  \[  \iint_{\Omega^\pm} \left( \left| \frac{\partial h}{\partial z} \right|^2 +
  \left| \frac{\partial h}{\partial \bar{z}} \right|^2 \right) dA_z < \infty  \]
  where $dA_z$ denotes Lebesgue measure with respect to $z$.
  On $\Omega^-$ we adopt the convention that $h(1/z)$ is harmonic in a neighbourhood of
  $0$.
   The Dirichlet space $\mathcal{D}(\Omega^\pm)$ is those elements of $\mathcal{D}(\Omega^\pm)_{\mathrm{harm}}$ which are holomorphic.   These spaces of functions are
  conformally invariant.
 \end{definition}
 We endow $\mathcal{D}(\Omega^+)_{\mathrm{harm}}$ with the norm
 \[  \| h\|^2 =  |h(0)|^2 +  \iint_{\Omega^\pm} \left( \left| \frac{\partial h}{\partial z} \right|^2 +
  \left| \frac{\partial h}{\partial \bar{z}} \right|^2 \right) dA_z < \infty  \]
 and similarly for $\mathcal{D}(\Omega^-)_{\mathrm{harm}}$, except that we replace
 $h(0)$ with $h(\infty)$.  We give $\mathcal{D}(\Omega^\pm)$ the restriction of this
 norm.

 We obtain boundary values of such functions in the following way.
 Fixing a point in $\Omega^+$ (say $0$), consider the set of hyperbolic
 geodesics $\gamma_\theta$ emanating from $0$ (equivalently, orthogonal
 curves to the level curves of Green's function with singularity at $0$).
 Each such geodesic terminates at a unique point on the boundary.  For
 any $h \in \mathcal{D}(\Omega^+)_{\mathrm{harm}}$ the limit of $h$ along
 such a ray exists for almost every ray \cite{Osborn, SchippersStaubach_Grunsky_quasicircle}.  This is independent
 of the choice of singular point.  In fact this can be done for any Jordan
 curve.  Call this set of boundary values $\mathcal{H}(\Gamma)$. Conversely,
 any element of $\mathcal{H}(\Gamma)$ has a unique harmonic extension with
 finite Dirichlet energy.  We endow $\mathcal{H}(\Gamma)$ with the norm
 inherited from $\mathcal{D}(\Omega^+)_{\mathrm{harm}}$. The set of
  boundary values of $\mathcal{D}(\Omega^+)_{\mathrm{harm}}$ obtained as above
  is identical with the space $\mathcal{H}(\Gamma)$ in the sense of $\mathrm{Definition}$ $\ref{de:sobolevboundary},$ if we identify $\Omega^+$ with the Riemann surface.

 Note however that this can be done by extending instead to $\Omega^-$.
  Given $h \in \mathcal{H}(\Gamma)$, let $h_{\Omega^\pm}$ denote the unique
  harmonic extensions to $\Omega^\pm$.
 For quasicircles, we have
 \begin{theorem}[\cite{SchippersStaubach_Osborn}]  Let $\Gamma$
  be a quasicircle not containing $0$ or $\infty$.
 There is a constant $C$ depending only on
  $\Gamma$ such that for any $h \in \mathcal{H}(\Gamma)$
  \begin{equation} \label{eq:out_in_comparison}
   \frac{1}{C} \| h_{\Omega^-} \| \leq \| h_{\Omega^+} \| \leq C \| h_{\Omega^-} \|.
  \end{equation}

  Conversely, if $\Gamma$ is a Jordan curve not containing $0$ and $\infty$
  such that the estimate
  $(\ref{eq:out_in_comparison})$ holds for some $C$, then $\Gamma$ is a quasicircle.
 \end{theorem}
 Thus we will not distinguish between the norms obtained from $\Omega^+$ and $\Omega^-$.

 \begin{remark}
  In the case of WP-class riggings, the boundary values exist in
   a much stronger sense \cite{RSS_WPjump}.  The boundary values are in
   a certain Besov space, which is norm equivalent to $\mathcal{H}(\Gamma)$, by Theorem 2.9 in \cite{RSS_WPjump}.
  \end{remark}

In connection to the boundedness of the operator $C_F$ the following theorem will be useful.

 \begin{theorem}\label{th:comp_isomorphism}  Let $\Omega$ and $D$ be
  quasidisks.
  If $F: \Omega \rightarrow D$ is a conformal bijection then
   $\mathcal{C}_F:\mathcal{H}(\partial D) \rightarrow \mathcal{H}(\partial\Omega)$
   is a bounded isomorphism.
  \end{theorem}
  \begin{proof}
   If we choose the norm obtained from the harmonic extensions to the inside, then
   this follows immediately from the conformal invariance of the Dirichlet integral.
   The claim thus follows for equivalent choices of norm.
  \end{proof}

  We can now turn to the jump decomposition.  General quasicircles are not rectifiable, so the
  Cauchy integral does not make sense on them directly.  Given a quasicircle $\Gamma$, let
  $f:\disk^+ \rightarrow \Omega^+$ and $g:\disk^- \rightarrow \Omega^-$ be conformal maps.
  Let $\gamma_r$ denote the circle $|w|=r$ traced counterclockwise.
  We define for $h \in \mathcal{H}(\Gamma)$ and $z \notin \Gamma$
  \begin{align}  \label{eq:limiting_integrals}
   J(\Gamma)h(z) & = \lim_{r \nearrow 1} \frac{1}{2 \pi i} \int_{f(\gamma_r)}
   \frac{h(\zeta)}{\zeta - z} d\zeta \nonumber \\
   & =  \lim_{r \searrow 1} \frac{1}{2 \pi i} \int_{g(\gamma_r)}
   \frac{h(\zeta)}{\zeta - z} d\zeta.
  \end{align}
  It is not obvious that one gets the same result using the maps $f$ and $g$.
  \begin{theorem}[\cite{SchippersStaubach_Grunsky_quasicircle}]  \label{th:decomposition}
   The limits $(\ref{eq:limiting_integrals})$ exist, are equal, and are independent
   of the choice of conformal map $f$ or $g$.  Furthermore, for any $h \in \mathcal{H}(\Gamma)$
   the restrictions of $J(\Gamma) h$ to $\Omega^\pm$ are in $\mathcal{D}(\Omega^\pm)$
   respectively.

   Finally, the map
   \begin{align*}
    K: \mathcal{H}(\Gamma) & \longrightarrow \mathcal{D}(\Omega^+) \oplus \mathcal{D}(\Omega^-) \\
    h & \longmapsto \left( \left. J(\Gamma) h \right|_{\Omega^+}, - \left. J(\Gamma) h \right|_{\Omega^-}
     \right)
   \end{align*}
   is a bounded isomorphism.
  \end{theorem}
  Denote the associated projection operators by
  \begin{align*}
   P(\Omega^{\pm}):\mathcal{H}(\Gamma) & \longrightarrow
   \mathcal{D}(\Omega^\pm) \\
   h & \longmapsto \pm \left. J(\Gamma) h \right|_{\Omega^\pm}.
  \end{align*}

  In the case of general quasicircles, it is as yet unknown whether
  \[  h(z) = P(\Omega^+) h (z) + P(\Omega^-) h(z)   \]
  for $z$ on the boundary in some reasonable limiting sense
  (see \cite{SchippersStaubach_Grunsky_quasicircle} for a specific conjecture).
  For WP-class quasicircles, we have the following result:

  \begin{theorem}[\cite{RSS_WPjump}]  Let $\Gamma$ be a $\mathrm{WP}$-class
   quasicircle.   For any $h \in \mathcal{H}(\Gamma)$,
   the boundary values of $h_\pm(z) = \pm P(\Omega^\pm) h(z)$ exist almost everywhere
   in a certain Besov space, and
   \[  h(z) = h_+(z) - h_-(z)  \]
   for almost all $z \in \Gamma$.
  \end{theorem}
  The boundary values are defined non-tangentially,
  and the jump agrees with $h$ almost everywhere.

\end{subsection}
\begin{subsection}{Segal-Wilson Grassmannian and Grunsky operators} \label{se:SW_Grunsky}

Work of A. A. Kirillov and D. V. Yuri'ev \cite{KY1,KY2} in the differentiable case,
further developed by S. Nag and D. Sullivan \cite{NagSullivan}, resulted in a representation
of the set of quasisymmetries of the circle modulo M\"obius transformations of the circle
 by a collection of polarizing subspaces in an infinite-dimensional Siegel disk.
The group $\qs(\mathbb{S}^1)/\text{M\"ob}(\mathbb{S}^1)$ can be canonically identified
with the universal Teichm\"uller space $T(\disk)$, so
this representation is analogous in some ways to the classical period mapping
of Riemann surfaces.
Furthermore, it is an example of the Segal-Wilson
universal Grassmannian, and thus relates to the determinant line bundle.  We refer to this
as the KYNS period mapping, following Takhtajan and Teo \cite{Takhtajan_Teo_Memoirs}.

Takhtajan and Teo \cite{Takhtajan_Teo_Memoirs} showed that this is a holomorphic map
into the set of bounded operators on $\ell^{2}$,
and that the restriction to the WP-class universal Teichm\"uller space is an
inclusion of the latter into the Segal-Wilson universal Grassmannian, and is also
a holomorphic mapping of Hilbert manifolds.   They also
gave an explicit form for the period map in terms of so-called Grunsky operators, which are
an important construction in geometric function theory relating to univalence and
quasiconformal extendibility (see e.g. \cite{Pommerenkebook}).  Furthermore, they
showed that a determinant of a certain operator related to the Grunsky matrix
is a K\"ahler potential for the Weil-Petersson metric.
On the other hand, work of the authors \cite{RSS_Teich_embed, SchippersStaubach_Grunsky_quasicircle} shows that the Grunsky operator
relates to the jump decomposition and to the operator $\pi$ (work in progess) whose determinant
line bundle we are concerned with.

We therefore review some of the material on the Segal-Wilson Grassmannian and
the Grunsky operator. We then place it in the context of the operator
$\pi$ and the Dirichlet space $\mathcal{D}(\riem)$ of a Riemann surface,
and relate it to the choice of the analytic class of riggings.  In the next section,
we also indicate how this relates to the problem of characterizing the cokernel
of $\pi$. \bigskip

Let $\mathscr{V}$ be an infinite-dimensional separable complex Hilbert space and let
$$\mathscr{V}=V_{+}\oplus V_{-}$$
be its decomposition into the direct sum of infinite-dimensional closed subspaces $V_{+}$ and $V_{-}$. The \emph{Segal-Wilson universal Grassmannian} $\mathrm{Gr}(\mathscr{V})$ \cite{Segal-Wilson} is defined as the set of closed subspaces $W$ of $\mathscr{V}$ satisfying the following conditions.
\begin{itemize}
\item[(1)] The orthogonal projection $\mathrm{pr}_{+}: W\rightarrow V_{+}$ is a Fredholm operator.
\item[(2)] The orthogonal projection $\mathrm{pr}_{-}: W\rightarrow V_{-}$ is a Hilbert-Schmidt operator.
\end{itemize}
Equivalently, $W\in\mathrm{Gr}(\mathscr{V})$, if $W$ is the image of an operator
$T: V_{+}\rightarrow W$ such that $\mathrm{pr}_{+} T$
is Fredholm and $\mathrm{pr}_{-} T$ is Hilbert-Schmidt.
The Segal-Wilson Grassmannian $\mathrm{Gr}(\mathscr{V})$ is a Hilbert manifold modeled on
the Hilbert space of Hilbert-Schmidt operators from $V_+$ to $V_-$.

As discussed in Subsection \ref{se:jump}, we have the natural decomposition
$\mathcal{H}(\mathbb{S}^1) = \mathcal{D}(\disk^+) \oplus \mathcal{D}(\disk^-)$.
From here on, we shall denote the elements
of the Dirichlet space $\mathcal{D}(\disk^\pm)$ that vanish at the origin or at infinity, and
the elements of $\mathcal{H}(\mathbb{S}^1)$ which have zero average value, by $\mathcal{D}_*(\disk^{\pm})$ and $\mathcal{H}_*(\mathbb{S}^1)$ respectively. The decomposition
therefore continues to hold with these changes of notation. We then make
the identifications $\mathscr{V} = \mathcal{H}_*(\mathbb{S}^1)$ and
$V_\pm = \mathcal{D}_*(\disk^{\mp})$.

The space $\mathcal{H}_*(\mathbb{S}^1)$ has a natural symplectic pairing, given by
\[  \left( g,h \right) =  \frac{1}{2\pi} \int_{\mathbb{S}^1} g \cdot dh.  \]
Given $h \in \mathcal{H}(\mathbb{S}^1 )$ and an orientation-preserving homeomorphism $\phi$ we define
\begin{equation}
 \hat{C}_\phi h = h \circ \phi - \frac{1}{2 \pi } \int_{\mathbb{S}^1} h \circ \phi. \end{equation}
(The second term restores the average to zero after composition).
We have the following theorem, formulated by Nag and Sullivan \cite{NagSullivan}.
\begin{theorem}[\cite{NagSullivan}] \label{th:NagSullivan_bounded} Let $\phi:\mathbb{S}^1 \rightarrow \mathbb{S}^1$ be an orientation-preserving
 homeomorphism.  $\hat{C}_\phi$ is bounded  from
 $\mathcal{H}_*(\mathbb{S}^1)$ to $\mathcal{H}_*(\mathbb{S}^1)$ if and only
 if $\phi$ is a quasisymmetry.  Furthermore, for any quasisymmetry, $\hat{C}_\phi$
 is a symplectomorphism.
\end{theorem}
We will show below that for each choice of WP-class quasisymmetry $\phi$,
$W = \hat{C}_\phi \mathcal{D}_*(\disk^-)$ is an element of the universal Grassmannian.
One can also relate the set of such subspaces to the set of Lagrangian decompositions
$\mathcal{H}_*(\mathbb{S}^1) = W \oplus \overline{W}$.  These decompositions
are acted on by the group of symplectomorphisms.
 For details see \cite{NagSullivan,Takhtajan_Teo_Memoirs}.

 We can give the decomposition $W \oplus \overline{W}$ the following interpretation.
 From Theorem \ref{th:welding}, consider the conformal welding decomposition $\phi = G^{-1} \circ F$ (suitably
 normalized), and let $\Omega^{-} = G (\disk^-)$ be identified with a Riemann surface
 and $F$ (restricted to the boundary) be identified with the rigging.  In that case,
 the decomposition $\mathcal{H}_*(\mathbb{S}^1) =
  W \oplus \overline{W}$ is the pull-back of the decomposition
   $\mathcal{H}(\Gamma) = \mathcal{D}(\Omega^-) \oplus \overline{\mathcal{D}(\Omega^-)}$
 under $F$ (modulo constants).  That is, it is the pull-back
  into holomorphic and anti-holomorphic parts of the decomposition of the
 complex harmonic Dirichlet space.
 This interpretation only becomes possible with our Plemelj-Sokhotski decomposition for quasicircles (at the very least for the WP-class quasicircles).

Theorem \ref{th:NagSullivan_bounded} above
gives yet another reason that quasisymmetries are a natural choice of riggings:
\noindent \begin{theorem} \label{th:quasisymmetries_largest}
 Let $\riem^B$ be a bordered Riemann surface of type $(g,n)$. The space
 $\mathcal{R}(\qs(\mathbb{S}^1),\riem^B)$ is the largest class of riggings
 such that $\mathcal{C}_{\phi_i} \mathcal{H}(\partial_i \riem^B)
 \subseteq \mathcal{H}(\mathbb{S}^1)$ and $\mathcal{C}_{\phi_i}$ is bounded
 for all $i=1,\ldots,n$ and $\phi=(\phi_1,\ldots,\phi_n) \in
 \mathcal{R}(\qs(\mathbb{S}^1),\riem^B)$.
\end{theorem}
\begin{proof}
 By Theorem \ref{th:qs_riggings_preserve_H} any $\phi \in \mathcal{R}(\qs(\mathbb{S}^1 ),\riem^B)$ has
 this property. Conversely fix a rigging $\phi=(\phi_1,\ldots,\phi_n) \in
 \mathcal{R}(\qs(\mathbb{S}^1),\riem^B).$ Now let $\psi = (\psi_1,\ldots,\psi_n)$
 be such that $\psi_i:\mathbb{S}^1 \rightarrow \partial_i \riem$ are
 orientation-preserving homeomorphisms, and
 $\mathcal{C}_{\psi_i}$ have images
 in $\mathcal{H}(\mathbb{S}^1)$ and
 are bounded.  Then $\mathcal{C}_{\phi_i^{-1} \circ \psi_i}:\mathcal{H}(\mathbb{S}^1)
 \rightarrow \mathcal{H}(\mathbb{S}^1)$ is bounded, so $\hat{\mathcal{C}}_{\phi_i^{-1} \circ \psi_i}$
 is bounded. By Theorem \ref{th:NagSullivan_bounded} $\phi_i^{-1} \circ \psi_i$
 is a quasisymmetry, which implies that $\psi_i =\phi_i\circ (\phi_i^{-1} \circ \psi_i )  \in \mathcal{H}(\partial_i \riem^B)$.
\end{proof}

Thus, if one accepts that $H^{1/2}(\partial_i \riem)$ is the correct set of boundary
values following Y.-Z. Huang \cite{Huang}, one sees that the quasisymmetries are the
set of possible riggings for which Fourier decompositions with
bounded projections exist.

Next, we need a trivialization of the Dirichlet space over the moduli space of all quasidisks.
\begin{theorem}[\cite{SchippersStaubach_Grunsky_quasicircle, ShenFaber}]  \label{th:DavidSheniso}
 Let $\Omega^-$ be a quasidisk containing $\infty$ and not containing $0$ in its closure. Choose $F \in \Oqc$ such that $\Omega^- = \hat{\mathbb{C}} \setminus \overline{F(\mathbb{D})}$. Then
 \[\mathbf{I}_F:=P(\Omega^-)\,\mathcal{C}_{F^{-1}}:\mathcal{D}_*(\disk^-) \longrightarrow \mathcal{D}_*(\Omega^-)  \]
 is a bounded linear isomorphism.
\end{theorem}
The striking fact that this is an isomorphism was first proven by Y. Shen \cite{ShenFaber},
where the operator was treated as acting on $\ell^2$.
Schippers and Staubach \cite{SchippersStaubach_Grunsky_quasicircle} gave the isomorphism
an interpretation in terms of projections and the Dirichlet space.

We now define the Grunsky operator.
\begin{definition}
  The Grunsky operator associated to $F \in \Oqc$ is
 \[   \mathrm{Gr}_F = P(\disk^+) \,\mathcal{C}_F \mathbf{I}_F:\mathcal{D}_*(\disk^-)  \longrightarrow \mathcal{D}_*(\disk^+).   \]
\end{definition}

What we give here requires that the projection $P(\Omega^-)$ is defined and bounded (Theorem \ref{th:decomposition}), the solution of the Dirichlet problem on
quasidisks with $\mathcal{H}(\Gamma)$ boundary data, the fact that boundary values
exist and are contained in $\mathcal{H}(\Gamma)$, and boundedness of the composition operator $\mathcal{C}_F$ \cite{SchippersStaubach_Grunsky_quasicircle}.  In the case of Weil-Petersson class riggings these results were established by the authors in \cite{RSS_WPjump}.
The Grunsky operator has several equivalent definitions, although a version
of the formula restricted to polynomials is classical.  An $L^2$ integral operator formulation
is due to Bergman and Schiffer and is also often
formulated on $\ell^2$ \cite{BergmanSchiffer,ShenGrunsky,Takhtajan_Teo_Memoirs}.
After work this can be shown to be equivalent for quasicircles
to the definition given here \cite{SchippersStaubach_Grunsky_quasicircle}.

The Grunsky operator was shown to be Hilbert-Schmidt precisely for WP-class
mappings $F$ by Y. Shen \cite{ShenGrunsky} and L. Takhtajan and L.-P. Teo independently \cite{Takhtajan_Teo_Memoirs}.
 Takhtajan and Teo also gave the connection to the
Segal-Wilson Grassmannian.  In our setting we can
give an explicit connection to the Dirichlet spaces of the domain $\Omega^-$,
where we treat $(\Omega^-,F)$ as a rigged surface.
\begin{theorem}
 Let $\phi$ be a $\mathrm{WP}$-class quasisymmetry, and $(F,G)$ the corresponding
 welding maps such that $\phi = G^{-1} \circ F$, $($we fix a normalization, say $F(0)=0, F'(0)=1, G(\infty)=\infty$$)$.
 Set $W = \mathcal{C}_{F} \mathcal{D}_*(\Omega^-) = \hat{\mathcal{C}}_\phi \mathcal{D}_*(\disk^-)$.
 We then have
 \begin{enumerate}
  \item $\mathcal{C}_F \mathbf{I}_F : \mathcal{D}_*(\disk^-)
  \rightarrow W$ is a bounded isomorphism;
  \item $P(\disk^-) \mathcal{C}_F \mathbf{I}_F = \mathrm{Id}$;
  \item $P(\disk^+) \mathcal{C}_F \mathbf{I}_F = \mathrm{Gr}_F$;
  \item $W$ is the graph of $\mathrm{Gr}_F$ in $\mathcal{H}_*(\mathbb{S}^1)$;
  \item $\mathrm{pr}_- = P(\disk^+): W \rightarrow \mathcal{D}_*(\disk^+)$ is  Hilbert-Schmidt;
  \item $\mathrm{pr}_+ = P(\disk^-): W \rightarrow \mathcal{D}_*(\disk^-)$ is Fredholm.
 \end{enumerate}
\end{theorem}
\begin{proof}
 The first claim follows from Theorem \ref{th:DavidSheniso}.  The second claim appears in
 \cite{SchippersStaubach_Grunsky_quasicircle}.  The third claim is the definition of
 $\mathrm{Gr}_F$.  The fourth claim follows from the previous three.  The fifth
 claim follows from (3), the fact that $\mathcal{C}_F \mathbf{I}_F$ is invertible and
 the fact that $\mathrm{Gr}_F$ is Hilbert-Schmidt, since Hilbert-Schmidt
 operators form an ideal.  The final claim follows in the same way from (2).
\end{proof}
 As a corollary, we have the following result.
 Recall that $\mathcal{H}_*(\mathbb{S}^1)$
 is set of elements of $\mathcal{H}(\mathbb{S}^1)$
  with zero average value.  The \emph{Shale group} or restricted general linear
 group corresponding to $\mathcal{H}_*(\mathbb{S}^1)$
 is the set of invertible operators
 \begin{equation} \label{eq:GLres_defin}
  \text{GL}_{\mathrm{res}}(\mathcal{H}_*(\mathbb{S}^1)) = \left\{
 \left( \begin{array}{cc} a & b \\ c & d \end{array} \right) \in \text{GL}(\mathcal{H}
  (\mathbb{S}^1)) \,:\, b,c \ \ \text{Hilbert-Schmidt} \  \right\}
 \end{equation}
  where the block decomposition is with respect to $\mathcal{H}_*(\mathbb{S}^1) =
  \mathcal{D}_*(\disk^+) \oplus \mathcal{D}_*(\disk^-)$.
 This implies that $a$ and $d$ are Fredholm \cite[Appendix D]{Huang_CFT}.
 $\text{GL}_{\mathrm{res}}(\mathcal{H}_*(\mathbb{S}^1))$ acts on the Segal-Wilson Grassmannian.
 \begin{theorem}
  The composition operators $\hat{C}_\phi$ corresponding
  to $\mathrm{WP}$-class quasisymmetries $\qs_{\mathrm{WP}}(\mathbb{S}^1)$
  are the completion of the analytic diffeomorphisms in the Shale group $\mathrm{GL}_{\mathrm{res}}(\mathcal{H}_*(\mathbb{S}^1))$.
 \end{theorem}
 \begin{proof} Let $\phi \in \qs(\mathbb{S}^1)$.
  The composition operator $\hat{C}_\phi$ has block decomposition
  \[  \hat{C}_\phi = \left( \begin{array}{cc} a & b \\ \bar{b} & \bar{a} \end{array}
   \right)  \]
    (\cite{Takhtajan_Teo_Memoirs}, see also \cite{SchippersStaubach_welding}).
   By \cite[Proposition D.5.3]{Huang_CFT} the composition operators corresponding
   to analytic diffeomorphisms $\phi$ embed in $\text{GL}_{\mathrm{res}}(\mathcal{H}_*(\mathbb{S}^1))$.
    By \cite{SchippersStaubach_welding}, $a$ is invertible, and Takhtajan
    and Teo showed that
    the Grunsky operator is given by $\overline{b} a^{-1}$  \cite{SchippersStaubach_welding, Takhtajan_Teo_Memoirs}. Since Hilbert-Schmidt
    operators form an ideal,  $\bar{b}$ and $b$
    are Hilbert-Schmidt if and only if the Grunsky operator is, which by
    Shen and Takhtajan-Teo holds if and only if $\phi \in \qs_{\mathrm{WP}}(\mathbb{S}^1)$.
 \end{proof}

The Grunsky operator and $\mathbf{I}_F$ are
 closely related to the operator $\pi$ defined in Section \ref{se:decompositions}.
We illustrate this in a very simple case, that of a Riemann surface biholomorphic to
the disk.  Let $F \in \Oqc$.  We identify the Riemann surface $\riem$ with the domain
$\Omega^-$ as above, and let $\left. F \right|_{\mathbb{S}^1}$ be the rigging.  Equivalently,
we consider the once-punctured Riemann sphere with puncture-type rigging $F$.

The following two results illustrate the connection (the second appears in \cite[Appendix D]{Huang} in the case of analytic rigging).
\begin{theorem}\label{Fiber identity} The operator $\pi$ corresponding to the rigged Riemann surface $(\Omega^-,F)$ satisfies
   $\mathcal{C}_F \circ  \pi \circ \mathbf{I}_F =\mathrm{Id}$. In particular, $\pi$ is a Fredholm operator with index equal to one.
 \end{theorem}
\begin{proof}
 The first claim follows from $\pi = \mathcal{C}_{F^{-1}} \circ P(\disk^-) \circ \mathcal{C}_F$,
 the definitions of $\mathbf{I}_F$ and the identity $P(\disk^-) \circ \mathcal{C}_F \circ
 P(\Omega^-) \circ \mathcal{C}_{F^{-1}} = \mathrm{Id}$ \cite{SchippersStaubach_Grunsky_quasicircle}.
 The cokernel is trivial, since it is isomorphic to
 \[  F^+(\partial \Omega^-)/\mathrm{Im}(\pi) = \mathcal{C}_{F^{-1}} \mathcal{D}_*(\disk^-)/\mathrm{Im}(\pi)
    \cong \mathcal{D}_*(\disk^-)/\mathcal{C}_F \mathrm{Im}(\pi). \]
 But $\mathrm{Im}(\mathcal{C}_F \circ \pi) = \mathrm{Im}(\mathcal{C}_F \circ \pi \circ \mathbf{I}_F) = \mathcal{D}_*(\disk^-)$
 where the first equality holds because $\mathbf{I}_F$ is an isomorphism, and the second follows from the
 fact that $\mathcal{C}_F \circ \pi \circ \mathbf{I}_F = \mathrm{Id}$.  The kernel is the space of constants functions, since for $h \in \mathcal{D}(\Omega^-)$ we have that $\pi h = 0$ if and only if $\mathcal{C}_{F^{-1}} P(\disk^-) \mathcal{C}_F h = 0$
 which is true if and only if $P(\disk^-) \mathcal{C}_F h = 0$.  This in turn holds if and only if $h \in \mathcal{D}_*(\Omega^+)$.
 By the uniqueness of the jump decomposition \cite{SchippersStaubach_Grunsky_quasicircle}, $h$ must
 be constant (since the constant function has the same jump as $h$).
\end{proof}

Finally, we make a few remarks on the condition that the off-diagonal
elements of $\mathrm{GL}_{\mathrm{res}}(\mathcal{H}_*(\mathbb{S}^1))$ are
Hilbert-Schmidt, and its relation to the determinant line bundle.
Details can be found in \cite[Appendix D]{Huang_CFT}.
One can obtain sections of the determinant line bundle of $\pi$
through the action of the central extension of $\mathrm{GL}_{\mathrm{res}}(\mathcal{H}_*(\mathbb{S}^1))$
on particular boundary curves of the Riemann surface. The central extension must have a
section $\sigma$ satisfying a cocycle condition.  Denoting elements of $\mathrm{GL}_{\mathrm{res}}(\mathcal{H}_*(\mathbb{S}^1))$ by $A_i$ and the elements of their
block decompositions by $a_i,b_i$, etc., this cocycle condition takes the form
\[  \sigma(A_1 A_2) = \mathrm{det}(a_1 a_2 a_3^{-1}) \sigma(A_3)  \]
for $A_3 = A_1 A_2$,  whenever $a_3$ is invertible \cite[p 257]{Huang_CFT}.  Recall that $a_3$ is invertible
for elements corresponding to composition by elements of $\qs_{\mathrm{WP}}(\mathbb{S}^1)$.
Multiplying the block matrices, we have $a_3 = a_1 a_2 + b_1 c_2$ so $a_1 a_2 a_3^{-1} = I - b_1 c_2 a_3^{-1}$.  Thus since trace-class operators
form an ideal, $a_1 a_2 a_3^{-1}$ has a determinant if and
only if $b_1 c_2$ is trace class, which occurs precisely when we require that
the off-diagonal elements $b_1$ and $c_2$ are Hilbert-Schmidt.  In other words,
the existence of a section of the central extension of $\mathrm{GL}_{\mathrm{res}}$
requires at a minimum the Hilbert-Schmidt condition on the off-diagonal elements.

Thus, we see that the natural analytic setting for conformal
field theory is the WP-class rigged moduli space.  However,
this leads to the following interesting question.  It appears that the determinant
line bundle of $\pi$ can be defined on the full quasiconformal rigged moduli space.  However,
the procedure above would only produce sections over leaves corresponding to perturbations
of fixed elements of the rigged moduli space by WP-class quasisymmetries.   Does this
extension have meaning within conformal field theory, and what is its interpretation
in Teichm\"uller theory?
\end{subsection}
\begin{subsection}{Grunsky operator and determinant line bundle}
 As we have stressed throughout the paper, the analytic issues are independent
 of the genus and number of boundary curves.  Nevertheless, one should
 demand that
 the techniques in the previous section ought to extend to higher genus $g$
 and number of boundary curves.  In this section we briefly indicate our results
 in this direction.

The constructions in the previous section can be generalized considerably.  Let
$\riem^P$ denote the Riemann sphere with $n$ fixed punctures.  Let $f = (f_1,\ldots,f_n) \in \Oqc(\riem^P)$ and let $\riem^B$ be the Riemann surface given by removing the closures
of $\cup_i f_i(\disk)$ from $\riem^P$.
The authors have shown that there is a natural generalization of the Grunsky
operator and the isomorphism $\mathbf{I}_f$ to maps
\[  \mathbf{Gr}_f : \bigoplus^n \mathcal{D}_*(\disk^-) \longrightarrow \bigoplus^n \mathcal{D}_*(\disk^+) \]
and
\[  \mathbf{I}_f: \bigoplus^n \mathcal{D}_*(\disk^-) \longrightarrow  \mathcal{D}(\riem^B).  \]
Also, we consider the map
\begin{align*}
 \mathcal{C}_f: \mathcal{D}(\riem^B) & \longmapsto \bigoplus^n \mathcal{H}_*(\mathbb{S}^1) \\
 h & \longmapsto \left( \left. h \circ f_1 \right|_{\mathbb{S}^1}, \cdots,
  \left. h \circ f_n \right|_{\mathbb{S}^1} \right).
\end{align*}
We then have the following result.
\begin{theorem}[\cite{RSS_Teich_embed}, \cite{RSS_Dirichletspace}]  \label{th:Grunsky_genus_zero}
The following statements hold$:$
 \begin{enumerate}
 \item  $\mathbf{I}_f$ is an isomorphism.
 \item Denote the
 space of bounded linear symmetric operators from
 $\bigoplus^n \mathcal{D}_*(\disk^-)$ to $\bigoplus^n \mathcal{D}_*(\disk^+)$ by $\mathcal{B}(n)$.
 There is a holomorphic injective map from the Teichm\"uller
 space of bordered surfaces of type $(0,n)$ into $T(\riem^P) \oplus \mathcal{B}(n)$.  The
 $\mathcal{B}(n)$ component of this map is given by $\mathbf{Gr}_f$.
 \item $\mathcal{C}_{f} \mathcal{D}(\riem^B)$ is the graph of the Grunsky operator
 $\mathbf{Gr}_f$.
\end{enumerate}
\end{theorem}
The image of $\mathbf{Gr}_f$ is into the closed unit ball, and we conjecture that
it maps into the open unit ball.
That is, we map into $T(\riem^P)$ times a generalization of the Siegel disk.
The map can be thought of as a generalization of the classical
period map of Riemann surfaces; indeed, one can further apply the period map
to $T(\riem^P)$ if desired.  It  is remarkable that  an idea from conformal field theory,
the rigged moduli space, makes it possible to realize a vast generalization
of a very old and important idea in complex analysis (related to the
  existence of projective embeddings, line bundles over the surface, etc.)
  Conversely, a classical object, the Grunsky operator,
sheds light on constructions in conformal field theory.

In the case of the Weil-Petersson class Teichm\"uller space, this map is into
a space of Hilbert-Schmidt operators.  It is obvious that a section
of the determinant line
bundle (in genus-zero) is easily constructed from this map (note however that one
must put the constants back in; see \cite{Huang_CFT}).

We would like to describe $\mathcal{C}_f \in \mathcal{D}(\riem^B)$ in the case
of surfaces of higher genus, to obtain local trivializations for use in the
construction of the determinant line bundle of $\pi$.  We also need
to describe the kernel and co-kernel of $\pi$ as well as the sewing isomorphism for determinant lines.  This requires
the jump formula and Cauchy kernel on Riemann surfaces of higher genus; in the case of analytic riggings Radnell \cite{Radnell_thesis}
successfully used these techniques to construct the determinant line bundle and sewing isomorphism, although an explicit description of the
cokernel was unnecessary for genus greater than one.

 Using results of H. Tietz  it is possible to extend the isomorphism $\mathbf{I}_f$
 and the Grunsky operator to higher genus; see \cite{Reimer_Schippers} (Reimer))
 for the case of tori with one boundary curve.
 Work in progress of the authors characterizes $\mathcal{D}(\riem)$ and the
 cokernel of $\pi$ for Riemann surfaces of genus $g$ with one boundary curve.
 One difficulty is dealing with the case that the punctures are Weierstrass points.
\end{subsection}
\end{section}
\begin{section}{Epilogue: why Weil-Petersson class riggings?}
 Now with many results for context, we can make a case for quasisymmetric
 and WP-class riggings. Note that the theorems quoted in this
 section are in some sense collections of theorems, and not all of them belong to the authors of the present paper. Proper references have been given in the main text.

 {\bf Why quasisymmetric riggings?}

 If we make the choice $\mathcal{R}(\qs(\mathbb{S}^1),\riem^B)$ or equivalently
 $\mathcal{R}(\Oqc,\riem^P)$ for riggings, we obtain the following advantages.
 \begin{enumerate}
  \item The rigged moduli space is a quotient of the Teichm\"uller space of the
   bordered surface (Theorem \ref{th:correspondence}). Thus, we see the connection
   between conformal field theory and Teichm\"uller theory.
  \item The rigged moduli space therefore inherits a complex Banach
  manifold structure, with
   respect to which sewing is holomorphic.  (Theorems \ref{th:correspondence}
    and \ref{th:sewing_holomorphic}).
  \item The quasisymmetries are the largest class of riggings preserving
   the $H^{1/2}$ space on the boundary (Theorem \ref{th:quasisymmetries_largest}).
  \item  The Plemelj-Sokhotski decomposition holds on quasicircles, and
   the projections to the holomorphic functions inside and outside are
   bounded (Theorem \ref{th:decomposition}).  Thus, using quasisymmetric riggings presents
   no analytic obstacles to
   investigating decompositions on sewn surfaces.
  \item One may describe the pull-back of holomorphic functions of finite
   Dirichlet energy (precisely the completion of analytic functions inside the
   $H^{1/2}$ space \cite{Huang}) on a genus zero surface as the graph of
   a Grunsky operator \ref{th:Grunsky_genus_zero}.
   Thus, it permits the investigation of the decompositions
   using classical techniques.
  \item Furthermore, through the lens of conformal field theory, one can see that the Grunsky
   operator can be used to extend the constructions of the classical period matrices for compact Riemann surfaces, to the case of non-compact surfaces.
 \end{enumerate}

 The quasisymmetric setting is sufficient to get a well-defined theory of holomorphic
 vector bundles.  In this paper, we have considered
 holomorphic functions of finite Dirichlet energy on a Riemann surface, which
 is a vector
 bundle over Teichm\"uller space; this generalizes
 readily to hyperbolically bounded holomorphic $n$-differentials.
 However, for the determinant line bundle, we require WP-class riggings.

 {\bf Why Weil-Petersson class riggings?}
 If we choose the riggings to be $\mathcal{R}(\qs_{\text{WP}}(\mathbb{S}^1),\riem^B)$, or equivalently
 $\mathcal{R}(\Oqc_{\text{WP}},\riem^P)$, then we obtain the following
 advantages.
 \begin{enumerate}
  \item The rigged moduli space is a quotient of the Weil-Petersson
   class Teichm\"uller space by a discrete group action (Theorem
    \ref{th:WP_mapping_class_prop}).
  \item The rigged moduli space is thus a complex Hilbert manifold
   (Theorem \ref{th:WP_Hilbert_manifold}).
   {\it Conjecture}: the sewing operation is holomorphic with respect
   to this manifold structure.
  \item The set of Weil-Petersson class maps of the circle $\qs_{\text{WP}}(\mathbb{S}^1)$ is a topological
  group (Theorem \ref{th:WP_topological_group}), whereas
  $\qs(\mathbb{S}^1)$ is not (Remark \ref{re:quasisymmetries_not_group}).
  \item The Weil-Petersson class riggings are the completion of the analytic diffeomorphisms of
 the circle in $\mathrm{GL}_{\mathrm{res}}(\mathcal{H}_*(\mathbb{S}^1))$, so that the action
 of the central extension of $\mathrm{GL}_{\mathrm{res}}(\mathcal{H}_*(\mathbb{S}^1))$
 produces sections of the determinant line bundle.
 \item It is the largest
 class of riggings for which $pr_- = P(\disk^+): W \rightarrow \mathcal{D}_*(\disk^+)\mathbf{}$
 is Hilbert-Schmidt in the Segal-Wilson universal Grassmannian.
 \item The Cauchy integral is defined and the jump decomposition holds on
  seam created by sewing by a WP-class quasisymmetry \cite{RSS_WPjump}.
 \end{enumerate}
\end{section}

\end{document}